\documentclass[10pt,conference]{ieeeconf}
\IEEEoverridecommandlockouts
\usepackage[utf8]{inputenc}

\usepackage[normalem]{ulem}
\usepackage{cite}
\usepackage{amsmath,amssymb,amsfonts,mathtools}
\usepackage{graphicx}
\usepackage[export]{adjustbox}
\usepackage[dvipsnames]{xcolor}
\usepackage{hyperref}
\usepackage{cancel}
\usepackage[font=small]{caption}
\usepackage{arydshln}
\usepackage{epsfig}
\usepackage{verbatim} 
\usepackage{tikz}
\usepackage{mathrsfs}
\usepackage{eufrak}
\usepackage{algorithm}
\usepackage{algpseudocode}

\usepackage{amsthm}

\newcommand{\reals}{\mathbb{R}}
\newcommand{\transpose}{^\textrm{\textnormal{T}}}
\DeclareMathOperator*{\argmin}{arg\,min}

\DeclareMathOperator*{\interior}{int}

\DeclareMathOperator{\gc}{getCBF}

\DeclareMathOperator*{\proj}{proj}

\newtheorem{theorem}{Theorem}
\newtheorem{assumption}{Assumption}
\newtheorem{lemma}{Lemma}
\newtheorem{example}{Example}
\newtheorem{definition}{Definition}

\newtheorem{proposition}{Proposition}
\newtheorem{remark}{Remark}
\newtheorem{problem}{Problem}

\usepackage{soul}

\newcommand{\nablaq}{{\nabla\hspace{-2pt}_q\hspace{0.5pt}}}
\newcommand{\nablav}{{\nabla\hspace{-2pt}_v\hspace{0.5pt}}}

\newcommand{\todo}[1]{\textcolor{red}{To Do: #1}}

\def\BibTeX{{\rm B\kern-.05em{\sc i\kern-.025em b}\kern-.08em
		T\kern-.1667em\lower.7ex\hbox{E}\kern-.125emX}}

 \newcommand{\regularversion}[1]{\iffalse #1 \fi}
 \newcommand{\extendedversion}[1]{{#1}}

\title{\fontsize{18}{20} \bf Compositions of Multiple Control Barrier Functions Under \\ Input Constraints}

\author{Joseph Breeden and Dimitra Panagou
	\thanks{This work was supported by the National Science Foundation Graduate Research Fellowship Program and the Fran\c{}ois Xavier Bagnoud Fellowship.}%
	\thanks{The authors are with the Department of Aerospace Engineering, University of Michigan, Ann Arbor, MI, USA. Email: \texttt{\{jbreeden,dpanagou\}@umich.edu}%
	}
}

\begin{document}
	
	\maketitle
	\thispagestyle{empty}
	
	\begin{abstract}
		This paper presents a methodology for ensuring that the composition of multiple Control Barrier Functions (CBFs) always leads to feasible conditions on the control input, even in the presence of input constraints.
		In the case of a system subject to a single constraint function, there exist many methods to generate a CBF that ensures constraint satisfaction. However, when there are multiple constraint functions, the problem of finding and tuning one or more CBFs becomes more challenging, especially in the presence of input constraints. This paper addresses this challenge by providing tools to 1) decouple the design of multiple CBFs, so that a CBF can be designed for each constraint function independently of other constraints, and 2) ensure that the set composed from all the CBFs together is a viability domain. Thus, a quadratic program subject to all the CBFs simultaneously is always feasible. The utility of this methodology is then demonstrated in simulation for a nonlinear orientation control system.
		
	\end{abstract}
	
	
	\section{Introduction}
	
	Control Barrier Functions (CBFs) are a control synthesis method for ensuring that system state trajectories always remain within some specified \emph{safe set} \cite{CBF_Tutorial}. In general, there may exist points within the safe set for which all feasible trajectories originating at these points will eventually exit the safe set. In such cases, one often seeks to construct a CBF so that a subset of the safe set, herein called the \emph{CBF set}, is rendered forward invariant, where the CBF set is a \emph{viability domain} (i.e. a controlled invariant set) for the given system dynamics and input bounds. 
	The problem of finding a CBF is equivalent to the problem of finding a description for such a viability domain.
	This in itself is a challenging problem, but for simple safe sets, i.e. safe sets that can be described as a sublevel set of a single \emph{constraint function}, several authors have proposed strategies to find CBFs as functions of the constraint function, including \cite{CBF_Tutorial,euler_lagrange_cbfs,Aircraft_backup_cbf,Recursive_CBF,Automatica} among others. 
	
	One open challenge in the CBF literature is that most works assume that the viability domain is sufficiently simple to be described by the zero sublevel (or superlevel) set of a single CBF. More complex viability domains could be expressed as the intersection of the zero sublevel sets of multiple CBFs{\color{black}, e.g. when each CBF represents a single obstacle in a cluttered environment}. This is sometimes achieved by taking a smooth \cite{CBFs_for_STL}, nonsmooth \cite{Nonsmooth_CBFs}, or adaptive \cite{adaptive_multi_cbf} maximum over several CBFs, or by simply applying multiple CBFs at once in a Quadratic Program (QP) \cite{ConstrainedRobots,Control_Sharing}. However, these approaches all assume that one is able to find a collection of CBFs whose CBF sets form a viability domain when intersected, or else the QP could become infeasible, as is illustrated {\color{black}by the two planar constraints} in Fig.~\ref{fig:example}. 
	Finding such a collection of CBFs is a much more challenging problem than finding a single CBF. The authors are not aware of any general algorithms analogous to \cite{CBF_Tutorial,euler_lagrange_cbfs,Recursive_CBF,Aircraft_backup_cbf,Automatica} for finding several CBFs at once, excepting learning approaches \cite{neural_high_relative_degree,robey2020learning}, which can only yield probabilistic safety guarantees.
	
	\begin{figure}
		\centering
		\begin{tikzpicture}
			\node[anchor=south west,inner sep=0] (image) at (0,0) {\includegraphics[width=0.43\columnwidth,clip,trim={1.2in, 0.1in, 1.3in, 0.3in}]{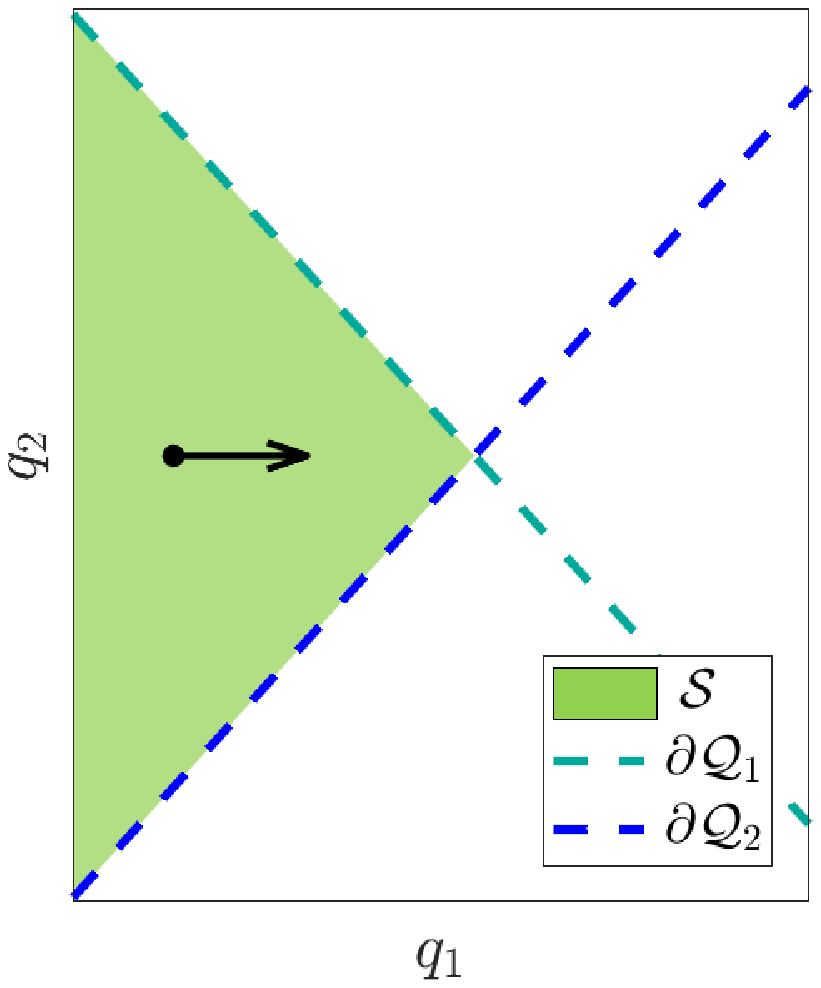}};
			\node[anchor=south east,inner sep=0] (image) at (\columnwidth,0) {\includegraphics[width=0.54\columnwidth,clip,trim={0.5in, 0.03in, 0.8in, 0.3in}]{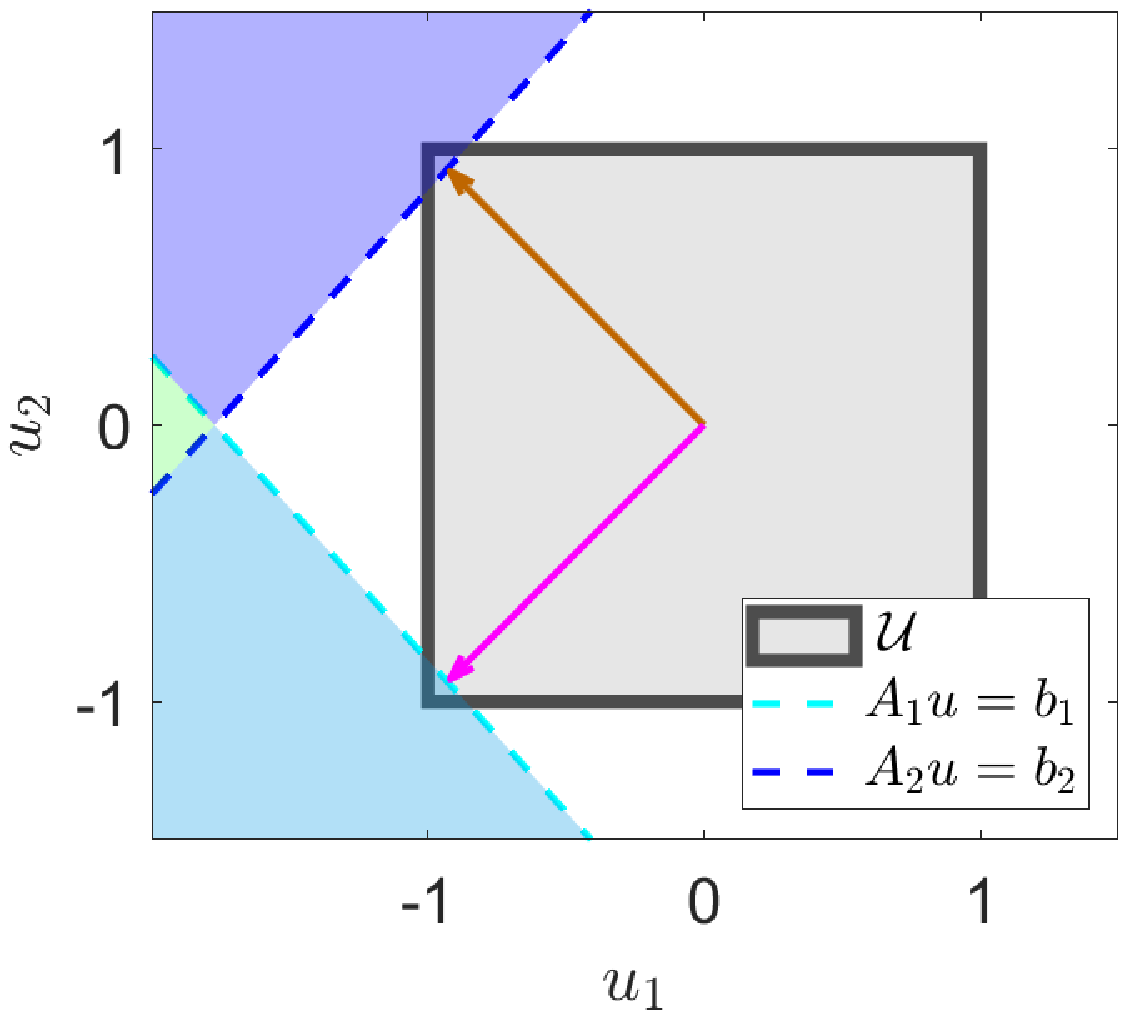}};
			\node [anchor=south west] (note) at (0.6,1.9) {$q$}; 
			\node [anchor=south west] (note) at (1.38,2.3) {$v$}; 
			\node [anchor=south west] (note) at (1.5,3.7) {\color{red}$\kappa_1 > 0$}; 
			\node [anchor=south west] (note) at (1.5,3.3) {\color{blue}$\kappa_2 \leq 0$};
			\node [anchor=south west] (note) at (1.2,0.8) {\color{JungleGreen}$\kappa_1 \leq 0$}; 
			\node [anchor=south west] (note) at (1.2,0.4) {\color{red}$\kappa_2 > 0$}; 
			%
			\node [anchor=south west] (note) at (0.31,2.85) 	{\small\color{JungleGreen}$\kappa_1\leq 0$}; 
			\node [anchor=south west] (note) at (0.31,2.5) {\small\color{blue}$\kappa_2\leq 0$}; 
			\node [anchor=south west] (note) at (2.45,2.3) {\color{red}$\kappa_1 > 0$}; 
			\node [anchor=south west] (note) at (2.45,1.9) {\color{red}$\kappa_2 > 0$};
			\node [anchor=south west] (note) at (4.28,4.0) {{CBF Conditions on $u$ at $(q,v)$}}; 
			\node [anchor=south west] (note) at (6,1.25) {\color{magenta}$w_1$}; 
			\node [anchor=south west] (note) at (6,3.05) {\color{BurntOrange}$w_2$}; 
			\node [anchor=south west] (note) at (4.45,3.65) {$A_2 u \leq b_2$}; 
			\node [anchor=south west] (note) at (4.45,0.67) {$A_1 u \leq b_1$}; 
		\end{tikzpicture}
		\caption{Left: Visualization of the safe positions ${q=(q_1,q_2)\in\reals^2}$ (green) for a double-integrator agent with two constraints $\kappa_1,\kappa_2$ as in Example~\ref{example}. 
		The state $x_0=(q,v)\in\reals^4$ in Example~\ref{example} has position $q$ labeled above and velocity $v$ pointing to the right. \\
		Right: Visualization of the control space at $x_0$ in Example~\ref{example}. Both CBF conditions are individually feasible via control inputs $w_1$ and $w_2$, but there is no $u\in\mathcal{U}$ satisfying both conditions simultaneously, so the intersection of the CBF sets is not a viability domain.}
		\label{fig:example}
		\vspace{-6pt}
	\end{figure}

	\begin{figure}
		\centering
		\begin{tikzpicture}
			\node[anchor=south west,inner sep=0] (image) at (0,0) 	{\includegraphics[width=0.49\columnwidth,clip,trim={0.68in, 0.85in, 0.5in, 0.7in}]{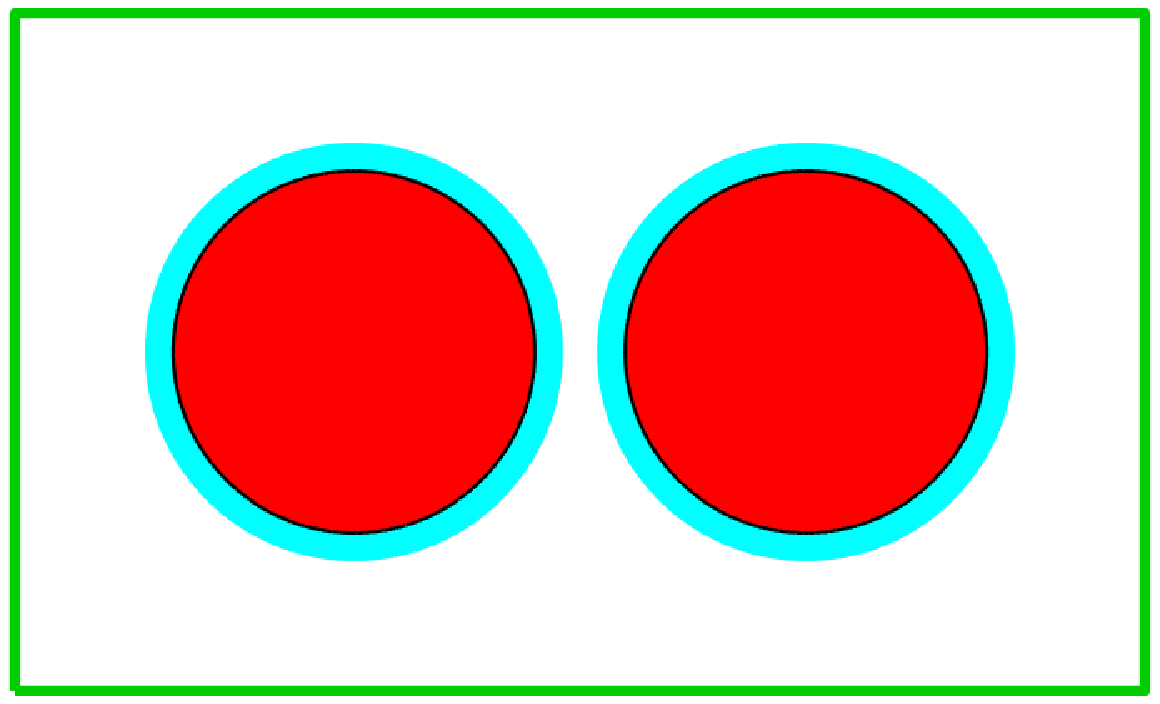}};
			\node[anchor=south east,inner sep=0] (image) at (\columnwidth,0) 	{\includegraphics[width=0.49\columnwidth,clip,trim={0.68in, 0.85in, 0.5in, 0.7in}]{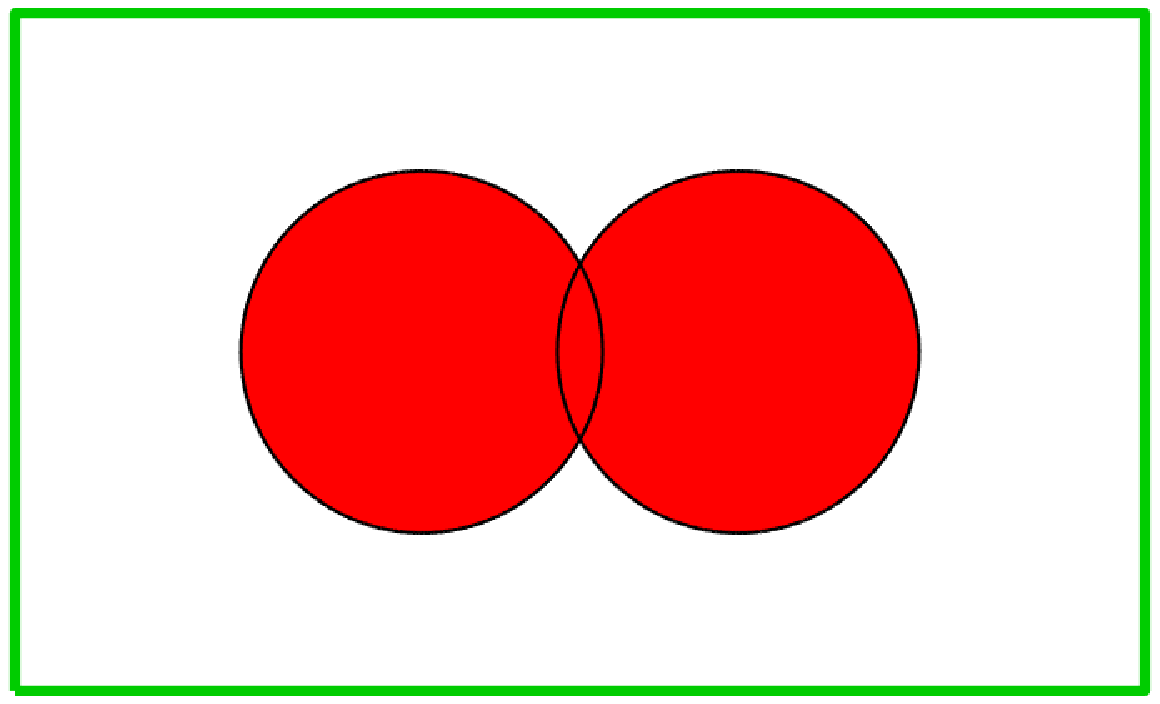}};
			\node [anchor=south west] (note) at (0.7,0.96) {$h_1 > 0$}; 
			\node [anchor=south west] (note) at (2.3,0.96) {$h_2 > 0$}; 
			\node [anchor=south west] (note) at (5.26,0.96) {$h_1 > 0$}; 
			\node [anchor=south west] (note) at (6.58,0.96) {$h_2 > 0$}; 
			\node [anchor=south west] (note) at (0.05,1.95) {\color{Green}$\mathcal{S}_1\cap\mathcal{S}_2$}; 
			\node [anchor=south west] (note) at (4.48,1.95) {\color{Green}$\mathcal{S}_1\cap\mathcal{S}_2$}; 
			\node [anchor=south west] (note) at (0.14,0.14) {\color{cyan}$h_1\geq-\epsilon$};
			\node [anchor=south west] (note) at (2.65,0.14) {\color{cyan}$h_2\geq-\epsilon$}; 
		\end{tikzpicture}
		\caption{Left: In \cite{type_ii_cbfs}, the problem of multi-CBF safety is simplified by assuming that $\{x \mid h_1(x)=0\}\cap\{x\mid h_2(x)=0\} = \emptyset$, and then designing separate safe controllers for the cases $h_1(x) \geq -\epsilon$ and $h_2(x) \geq -\epsilon$. Right: By contrast, this work is interested in the case where $\{x \mid h_1(x)=0\}\cap\{x\mid h_2(x)=0\} \neq \emptyset$, and thus both CBF conditions must be simultaneously satisfied, potentially resulting in conflicts such as that in Fig.~\ref{fig:example}.}
		\label{fig:boundaries}
		\vspace{-10pt}
	\end{figure}
	
	The most common strategy to employ multiple CBFs is to assume that all the CBFs act independently of each other \cite{safe_euler_lagrange,ACC2022,type_ii_cbfs}. For example, the CBFs may apply to different states with decoupled input channels \cite{safe_euler_lagrange,ACC2022}, or it may be the case that only one CBF acts at a time \cite{type_ii_cbfs}. The latter is equivalent to assuming that the boundaries of the individual CBF zero sublevel sets, i.e. the CBF zero level sets, do not intersect, as shown in Fig.~\ref{fig:boundaries}. In this case, the CBFs may be designed in a one-at-a-time fashion using existing algorithms. However, if this does not hold, then the CBFs must be designed all-at-once, or else the intersection of the CBF sets may not be a viability domain.
	
	There exist many tools for the computation of viability domains in the control verification literature 
	\cite{power_systems,zonotopes_sampled_data,viability_definitions,expanding_polytope_ci,parameterized_linear_rci,viability_kernel_recrusive_improved,barrier_domain_of_attraction}. Such tools have also been used to construct CBFs \cite{HJ_CBFs}, and have been augmented via application of CBF concepts \cite{viable_sets_barrier_func_improved}. 
	However, all of these algorithms are computationally expensive, even for linear systems \cite{parameterized_linear_rci,viability_kernel_recrusive_improved}.
	
	Compared to prior works, this paper presents two contributions. First, we present a methodology for decoupling the design of multiple CBFs in the presence of prescribed input bounds. This decoupling allows one to leverage existing single-CBF algorithms \cite{CBF_Tutorial,euler_lagrange_cbfs,Recursive_CBF,Aircraft_backup_cbf,Automatica} while avoiding conflicts such as those in Fig.~\ref{fig:example}. Second, we present an iterative algorithm to find a viability domain parameterized by these decoupled CBFs. That is, instead of using zonotopes, \cite{viability_definitions,zonotopes_sampled_data}, polytopes \cite{expanding_polytope_ci,parameterized_linear_rci}, or other parameterized functions \cite{barrier_domain_of_attraction,viable_sets_barrier_func_improved,viable_sublevel_sets}, this paper expresses viability domains in terms of an intersection of CBF sets. As each CBF forbids state trajectories from crossing the boundary of its own CBF set, our algorithm focuses on verifying Nagumo's condition only at the states where the boundaries of multiple CBF sets intersect.
	One can then compute safe control actions using a quadratic program (QP) \cite{CBF_Tutorial,Choice_of_class_k} subject to the CBF conditions arising from every CBF, and we show that this QP is always feasible. Note that, compared to \cite{power_systems,zonotopes_sampled_data,viability_definitions,expanding_polytope_ci,parameterized_linear_rci,viability_kernel_recrusive_improved,barrier_domain_of_attraction}, the viability domains resulting from this algorithm will generally be more conservative, as our intent is to enable the use of existing (usually conservative) CBF literature rather than to approximate the maximal viability kernel.

	
	

	
	
	
	\section{Preliminaries} \label{sec:preliminaries}
	
	
	
	%
	
	\subsection{Notations}
	 
	Given a set $\mathcal{S}$, let $\partial \mathcal{S}$ denote the boundary of $\mathcal{S}$, and $\interior(\mathcal{S})$ denote the interior of $\mathcal{S}$. Given a function $\psi:\reals^{n_1}\times\reals^{n_2}$, denoted $\psi(q,v)$, let $\nablaq \psi(q,v)$ denote the row vector of derivatives (i.e. the gradient) of $\psi$ with respect to inputs $q\in\reals^{n_1}$, and let $\nablav \psi(q,v)$ denote the row vector of derivatives of $\psi$ with respect to inputs $v\in\reals^{n_2}$. Let $\mathcal{C}^r$ denote the set of functions $r$-times continuously differentiable in all arguments, and $\mathcal{K}$ the set of class-$\mathcal{K}$ functions. Let $a \cdot b$ denote the dot product between vectors $a$ and $b$. {\color{black}Let $a\transpose$ denote the transpose of the vector $a$.} Let $\|\cdot\|_1$, $\|\cdot\|_2$, and $\|\cdot\|_\infty$ denote the 1-norm, 2-norm, and $\infty$-norm respectively. Let $[N]$ denote the set of all integers between 1 and $N$. Given one or more vectors $\{v_1,\cdots,v_N\}$, define $\proj_{\{v_1,\cdots,v_N\}} u = \sum_{i\in[M]} (u \cdot \hat{b}_i) \hat{b}_i$, where $\{\hat{b}_1,\cdots,\hat{b}_M\}$ is an orthonormal basis spanning $\textrm{span}\{v_1,\cdots,v_N\}$.
	
	\subsection{Model}
	
	Let $q \in \reals^{n_1}$ be the coordinates, and $v \in \reals^{n_2}$ the velocities, of a second-order system of the form
	\begin{subequations}
	\begin{align}
		\dot{q} &= g_1(q) v ,\, \\
		\dot{v} &= f(q,v) + g_2(q) u \,, 
	\end{align}\label{eq:model}%
	\end{subequations}
	with state $x \triangleq (q,v) \in\reals^{n_1+n_2}$ and control input $u\in\mathcal{U}\subset\reals^m$. Let $f:\reals^{n_1}\times\reals^{n_2}\rightarrow\reals^{n_2}$, $g_1:\reals^{n_1}\rightarrow\reals^{n_1\times n_2}$, and $g_2:\reals^{n_1}\rightarrow\reals^{n_2\times m}$ belong to $\mathcal{C}^2$.
	Let $\mathcal{U}$ encode the set of allowable control inputs, herein called the \emph{control set}, and assume that $\mathcal{U}$ is compact and convex and contains the zero vector. 
	Assume that {\color{black}$f$, $g_1$, $g_2$, and $u$ are sufficiently regular that} trajectories of \eqref{eq:model} exist and are unique for all times $t\in\mathcal{T}=[t_0,t_f)$, where $t_f$ is possibly $\infty$. Note that the model \eqref{eq:model} includes Euler-Lagrange systems, where $n_1=n_2$, as in \cite{euler_lagrange_cbfs,safe_euler_lagrange,safe_kinematic}, but is also more general.
	
	
	\subsection{Safety Definitions} \label{sec:safety}
	
	A principal requirement of any autonomous system should be to at all times satisfy certain operation constraints. 
	Specifically, suppose we are given several \emph{constraint functions} $\kappa_i:\reals^{n_1}\rightarrow \reals, \kappa_i\in\mathcal{C}^2, i \in[N_1]$ and $\eta_j:\reals^{n_2}\rightarrow\reals, \eta_j\in\mathcal{C}^1, j\in[N_2]$, which each generate a \emph{constraint set}:
	\begin{gather}
		\mathcal{Q}_i \triangleq \{ q \in \reals^n \mid \kappa_i(q) \leq 0 \} \,, \label{eq:q_set} \\
		\mathcal{V}_j \triangleq \{ v \in \reals^n \mid \eta_j(v) \leq 0\} \,. \label{eq:qdot_set}
	\end{gather}
	We call the composition of all the constraint sets the \emph{safe set} $\mathcal{S}$, and say that the system is \emph{safe} at time $t$ if $x(t)\in\mathcal{S}$:
	\begin{equation}
		%
		%
		\mathcal{S} \triangleq \left( \array{@{}c@{}}\displaystyle \bigcap_{i\in[N_1]} \endarray \mathopen{\raisebox{2pt}{$\mathcal{Q}_i$}} \right) \times \left( \array{@{}c@{}}\displaystyle \bigcap_{j\in[N_2]}\endarray \mathopen{\raisebox{2pt}{$\mathcal{V}_j$}} \right) \,. \label{eq:safe_set}
	\end{equation}
	That is, we allow for both position constraints $\kappa_i$ and velocity constraints $\eta_i$, but assume that these constraints are encoded separately. This implies that each $\kappa_i$ is of relative-degree 2 along \eqref{eq:model} and each $\eta_i$ is of relative-degree 1 along \eqref{eq:model}. This separation of $\mathcal{Q}_i$ and $\mathcal{V}_i$ is not necessary, but is greatly simplifying, as will be explained in Remark~\ref{remark:high_degree}.

	We now introduce two notions of CBF.
	
	\begin{definition} \label{def:cbf}
		Let $\mathcal{X}\subseteq\mathcal{S}$ and $\mathcal{Y}\subseteq\mathcal{U}$. A function $h:\reals^{n_1+n_2}\rightarrow\reals,h\in\mathcal{C}^1$ is a \emph{control barrier function (CBF) for $(\mathcal{X},\mathcal{Y})$} if there exists $\alpha\in\mathcal{K}$ such that the set
		\begin{equation}
			\boldsymbol\mu(q,v,\mathcal{Y}) \triangleq \{ u \in \mathcal{Y} \mid\dot{h}(q,v,u) \leq \alpha(-h(q,v)) \} \label{eq:control_set} \hspace{-1pt}
		\end{equation}
		is nonempty for all $(q,v)\in\mathcal{H}\cap\mathcal{X}$, where
		\begin{equation}
			\mathcal{H} \triangleq \{(q,v)\in\reals^{n_1+n_2}\mid h(q,v)\leq 0\} \,. \label{eq:cbf_set}
		\end{equation}
	 	We call $\mathcal{H}$ the \emph{CBF induced set}, or simply \emph{CBF set}.
	\end{definition}

	Note that under the dynamics \eqref{eq:model}, $\dot{h}$ is affine in $u$,
	\begin{equation}
		\dot{h}(q,v,u) = \nablaq h(q,v)g_1(q)v + \nablav h(q,v) g_2(q) u \,, \label{eq:control_affine}
	\end{equation}
	so {\color{black}if $\mathcal{Y}$ is convex, then} the set $\boldsymbol\mu$ in \eqref{eq:control_set} is {\color{black}also} convex.
	
	\begin{definition}\label{def:scbf}
		Let $\mathcal{X}\subseteq\mathcal{S}$ and $\mathcal{Y}\subseteq\mathcal{U}$. A function $h:\reals^{n_1+n_2}\rightarrow\reals,h\in\mathcal{C}^1$ is a \emph{simple control barrier function (SCBF) for $(\mathcal{X},\mathcal{Y})$} if 1) $h$ is a CBF for $(\mathcal{X},\mathcal{Y})$, and 2) the set \vspace{-12pt}
		\begin{multline}
			\hspace{12pt}\boldsymbol\mu^s(q,v,\mathcal{Y}) \triangleq \{ u \in \boldsymbol\mu(q,v,\mathcal{Y}) \mid \\ \exists c \geq 0 : u = {\color{black}-} c[\nablav h(q,v)g_2(q)]{\color{black}\transpose} \} \label{eq:scbf_control_set}
		\end{multline}
		is nonempty for all $(q,v)\in\mathcal{H}\cap\mathcal{X}$. 
	\end{definition}
	
	That is, $h$ is an SCBF if the condition $\dot{h}(q,v,u) \leq \alpha(-h(q,v))$ can always be satisfied for a control input $u$ {\color{black}anti-}parallel to the vector $\nablav h(q,v)g_2(q)$. We introduce the notion of SCBF because it allows for slightly less conservative viability domain computations, as we will show in Section~\ref{sec:methods}. \extendedversion{SCBFs arise naturally in many practical applications.} 
	Note that if $\mathcal{Y} = \{ u \in \reals^m \mid \|u\|_2 \leq u_\textrm{max}\}$, then any CBF for $(\mathcal{X},\mathcal{Y})$ is automatically an SCBF for $(\mathcal{X},\mathcal{Y})$ too. Next, we recall the definition of \emph{viability domain}.
	
	\begin{definition}[{\hspace{-0.5pt}\cite[Eq. 7]{viability_definitions}}]\label{def:viable}
		A set $\mathcal{A}\subseteq\mathcal{S}$ is called a \emph{viability domain} if for every point $x(t_0)\in\mathcal{A}$, there exists a control signal $u(t)\in\mathcal{U},t\in\mathcal{T}$ such that the trajectory $x(\cdot)$ of \eqref{eq:model} satisifes $x(t)\in\mathcal{A}$ for all $t\in\mathcal{T}$.
	\end{definition}
	
	Note that if 1) $h$ is a CBF for $(\mathcal{S},\mathcal{U})$ and 2) $\mathcal{H}\subseteq\mathcal{S}$, then it follows from Definition~\ref{def:cbf} and \cite[Thm.~2]{CBF_Tutorial} that $\mathcal{H}$ is a viability domain. This is useful because generally $\mathcal{S}$ in \eqref{eq:safe_set} is not a viability domain, but we can use CBFs to describe subsets of $\mathcal{S}$ that are viability domains. However, in this paper, we assume that $\mathcal{S}$ is sufficiently complex that it is difficult to find a single CBF $h$ for which $\mathcal{H}\subseteq\mathcal{S}$, thus motivating the following extensions of \cite[Thm.~2]{CBF_Tutorial}.	
	
	\subsection{Safe Quadratic Program Control}
	
	The heart of safety-critical control is Nagumo's theorem:
	\begin{lemma}[{\hspace{0.5pt}\cite[Thm.~3.1]{Blanchini}}] \label{lemma:nagumo}
		Let $\mathcal{A}\subseteq\reals^n$ be a closed convex set, and let $\mathcal{A}^\mathcal{T}(x)$ be the tangent cone of $\mathcal{A}$ at $x\in\reals^n$. Then $\mathcal{A}$ is forward invariant under \eqref{eq:model} if and only if $\dot{x}\in\mathcal{A}^\mathcal{T}(x)$ for all $x\in\mathcal{A}$.
	\end{lemma}
	\noindent
	Note that if $x\in\interior(\mathcal{A})$, then $\mathcal{A}^\mathcal{T}(x) = \reals^{n}$, so Lemma~\ref{lemma:nagumo} in effect only depends on the flow $\dot{x}$ when $x \in \partial\mathcal{A}$. We use this fact in Lemma~\ref{lemma:all_invariance} below.
	
	Given multiple constraint functions $\kappa_i, {i\in[N_1]}$ and $\eta_j, {j\in[N_2]}$, it is common \cite{ConstrainedRobots,Control_Sharing,compatibility_checking} to construct multiple CBFs $h_k,{k\in[M]}$. These can then be used for safe control according to the following lemma, which follows immediately from \cite[Thm. 2]{CBF_Tutorial}.
	
	\begin{lemma} \label{lemma:multi_invariance}
		Given functions $\{h_k\}_{k\in[M]}$, {\color{black}with $\mathcal{H}_k$ as in \eqref{eq:cbf_set},} denote $\mathcal{H}_\textrm{all}=\cap_{k\in[M]}\mathcal{H}_k$. Assume that each $h_k$ is a CBF for $(\mathcal{H}_\textrm{all},\mathcal{U}$). 
		Then any control {\color{black}law $u:\reals^{n_1+n_2}\rightarrow\mathcal{U}$} satisfying
		\begin{align}
			%
			%
			u(q,v)&\in\boldsymbol\mu_\textrm{all}(q,v)\triangleq \bigcap_{k\in[M]}\boldsymbol\mu_k(q,v,\mathcal{U}) \label{eq:cbf_condition}
		\end{align}%
		 for all $(q,v)\in\mathcal{H}_\textrm{all}$ will render $\mathcal{H}_\textrm{all}$ forward invariant. 
	\end{lemma}

	Note that Lemma~\ref{lemma:multi_invariance} does not include the safe set $\mathcal{S}$ from \eqref{eq:safe_set}, so safety is only guaranteed if $\mathcal{H}_\textrm{all}\subseteq\mathcal{S}$. We next introduce a modified version of Lemma~\ref{lemma:multi_invariance} to handle the possibility that $\mathcal{H}_\textrm{all} \not\subseteq\mathcal{S}$, as frequently occurs when using High Order CBFs (\regularversion{\color{black}HOCBFs, }e.g. \cite[Thm.~5]{Recursive_CBF},\cite[Lemma~7]{Automatica}).
	
	\begin{lemma} \label{lemma:all_invariance}
		Given functions $\{h_k\}_{k\in[M]}$, {\color{black}with $\mathcal{H}_k$ as in \eqref{eq:cbf_set},} denote $\mathcal{H}_\textrm{all}=\cap_{k\in[M]}\mathcal{H}_k$ and $\mathcal{A} = \mathcal{H}_\textrm{all}\cap\mathcal{S}$. Assume that each $h_k$ is a CBF for $(\mathcal{A},\mathcal{U})$. {\color{black}Let $u:\reals^{n_1+n_2}\rightarrow\mathcal{U}$ be a control law.} For {\color{black} every $i\in[N_1]$, $j\in[N_2]$, and} all $(q,v) \in \interior(\mathcal{H}_\textrm{all})$, assume that $\kappa_i(q) = 0 \implies \dot{\kappa}_i(q,v)\leq 0$ and that $\eta_j(v)=0\implies\dot{\eta}_j(q,v,u{\color{black}(q,v)})\leq 0$. 
		{\color{black}If $u$ satisfies \eqref{eq:cbf_condition} for all $(q,v)\in\mathcal{A}$, then $u$ will render $\mathcal{A}$ forward invariant.} 
	\end{lemma}

	Note that  $\kappa_i(q) = 0 \implies \dot{\kappa}_i(q,v)\leq 0$ is equivalent to Nagumo's necessary condition in Lemma~\ref{lemma:nagumo}. Thus, Lemma~\ref{lemma:all_invariance} highlights how one purpose of the CBFs $h_k$ is to construct a set $\mathcal{A}\subseteq\mathcal{S}$ that excludes all the states in $\mathcal{S}$ where Nagumo's necessary condition does not hold for any $u\in\mathcal{U}$. \extendedversion{This idea is the central motivation for the algorithm in Section~\ref{sec:three_constraints}.}
	
	Finally, given a collection of constraints $\kappa_i,\eta_j$ and CBFs $h_k$ satisfying the conditions of Lemma~\ref{lemma:all_invariance}, it is common to \regularversion{\color{black}use}\extendedversion{construct} {\color{black}control laws $u:\reals^{n_1+n_2}\rightarrow\mathcal{U}$} of the form \cite[Sec.~II-C]{CBF_Tutorial}%
	\begin{subequations}
	\begin{align}
		&u(q,v) = \argmin_{u\in\mathcal{U}} \|u - u_\textrm{nom}(q,v)\|_2^2 \label{eq:first_qp_norm} \\
		&\;\;\;\;\;\;\;\;\;\; \textrm{ s.t. } \dot{h}_k(q,v,u) \leq \alpha_k(-h_k(q,v)), \forall k\in[M] \label{eq:first_qp_conditions}
	\end{align}\label{eq:first_qp}%
	\end{subequations}
	where $\alpha_k$ comes from \eqref{eq:control_set}, $u_\textrm{nom}$ is any control law, and \eqref{eq:first_qp_conditions} is affine due to \eqref{eq:control_affine}. 
	If $u$ in \eqref{eq:first_qp} always exists, i.e. if $\boldsymbol\mu_\textrm{all}$ in \eqref{eq:cbf_condition} is nonempty for all $(q,v)\in\mathcal{A}$, then it follows from Lemma~\ref{lemma:all_invariance} that the set $\mathcal{A}$ in Lemma~\ref{lemma:all_invariance} is a viability domain. 
	However, if there exists $(q,v)\in\mathcal{A}$ for which $\boldsymbol\mu_\textrm{all}(q,v)$ is empty, then trajectories originating at $(q,v)$ might exit the safe set $\mathcal{S}$. Thus, the goal of this paper is to present tools to ensure that $\mathcal{A}$ is a viability domain, so that the related controller \eqref{eq:the_qp} (to be introduced) is always feasible.
	
	
	\regularversion{\vspace{-2pt}}
	\begin{problem} \label{problem}
		Determine a set of CBFs $\{h_k\}_{k\in[M]}$ such that the set $\mathcal{A}$ in Lemma~3 is a viability domain.
	\end{problem}
	
	\subsection{Assumptions and Motivating Example}
	
	The principal challenge addressed in this paper is the problem of multi-CBF compositions. For this reason, we assume that the single-CBF problem is sufficiently solved.
	

	\regularversion{\vspace{-2pt}}
	\begin{assumption}\label{assumption}
		Given sets $\mathcal{X}\subseteq\mathcal{S}$ and $\mathcal{Y}\subseteq\mathcal{U}$, focus on any single constraint function $\kappa_i$ (or $\eta_j$). Denote $\mathcal{O} = \mathcal{Q}_i\times\reals^{n_2}$ (or $\mathcal{O} = \reals^{n_1}\times\mathcal{V}_j$). Consider the set $\mathcal{Z} = \mathcal{X}\cap(\interior(\mathcal{X})\cup\partial\mathcal{O})$. Assume that there exists an algorithm (e.g. \cite{CBF_Tutorial,euler_lagrange_cbfs,Recursive_CBF,Aircraft_backup_cbf,Automatica}) to derive one or more functions $\{h_k\}_{k=k_1}^{k_2}$, each a CBF for $(\mathcal{X},\mathcal{Y})$, such that $\kappa_i(q)=0\implies \dot{\kappa}_i(q,v) \leq 0$ (or $\eta_j(v)=0\implies \dot{\eta}_j(q,v,u) \leq 0$) for all $(q,v)\in(\mathcal{Z}\cap(\cap_{k=k_1}^{k_2}{\color{black}\interior}(\mathcal{H}_k)))$ {\color{black}and all $u\in\mathcal{U}$, where $\mathcal{H}_k$ is as in \eqref{eq:cbf_set}}. That is, for each constraint function, assume that we can find one or more CBFs that prevents state trajectories from violating that particular constraint function.
		
	\end{assumption}

	\extendedversion{
	\begin{remark}
		In practice, for the relative-degree 1 constraint functions $\eta_j$, we can often choose a CBF $h_k$ so that $h_k = \eta_j$. In this case, the set $\mathcal{Z}\cap \interior(\mathcal{H}_k)$ in Assumption~\ref{assumption} has empty intersection with the set $\{(q,v)\in\mathcal{X} \mid \eta_j(v) = 0\}$, so the condition ``$\eta_j(v) = 0 \implies \dot{\eta}(q,v,u) \leq 0$ for all $(q,v)\in(\mathcal{Z}\cap\interior(\mathcal{H}_k)), u\in\mathcal{U}$'' is automatically satisfied. Therefore, the ``all $u\in\mathcal{U}$'' part of Assumption~\ref{assumption} is a rarely used technicality. One only needs to check this technicality if one chooses a CBF $h_k$ such that there exists $(q,v)\in\mathcal{H}_k$ where $\eta_j(v) > 0$. The authors are unaware of a practical example where this occurs for a relative degree 1 constraint function $\eta_j$, though such a choice of $h_k$ is common for relative-degree 2 constraint functions $\kappa_i$, e.g. \cite{Automatica}.
	\end{remark}
	}

	Successive application of Assumption~\ref{assumption} to every constraint function one-at-a-time then produces a collection of CBFs $\{h_k\}_{k\in[M]}$ satisfying the assumptions of Lemma~\ref{lemma:all_invariance}. However, this still does not imply joint feasibility of all the CBFs $h_k$, as we illustrate with the following example.

	\regularversion{\vspace{-2pt}}
	\begin{example} \label{example}
		Consider the 2D double integrator $\dot{q}=v, \dot{v} = u$, $q = (q_1,q_2)\in\reals^2,v=(v_1,v_2)\in\reals^2$, $u =(u_1,u_2)\in \mathcal{U} = [-1, 1]\times [-1,1]$ subject to two constraint functions $\kappa_1(q) = q_1 + \gamma q_2$ and $\kappa_2(q) = q_1 - \gamma q_2$ for some constant $\gamma > 0$, resulting in the safe set $\mathcal{S}=(\mathcal{Q}_1\cap\mathcal{Q}_2)\times\reals^2$, pictured in Fig.~\ref{fig:example}. From \cite{Automatica,Recursive_CBF}, one can derive CBFs $h_1,h_2$ for $(\mathcal{S},\mathcal{U})$, where $h_i(q,v) = \dot{\kappa}_i(q,v) - \sqrt{-2(1+\gamma)\kappa_i(q)}$, that satisfy the conditions of Lemma~\ref{lemma:all_invariance}. Denote $\mathcal{A}=\mathcal{H}_1\cap\mathcal{H}_2\cap\mathcal{S}$ and let $x_0 = (q,v) = (-\frac{1}{2(1+\gamma)},0,1,0) \in \partial\mathcal{A}$. Then there is no $u \in \mathcal{U}$ that renders $\mathcal{A}$ forward invariant from $x_0$ (see the right side of Fig.~\ref{fig:example}). That is, Nagumo's necessary condition (Lemma~\ref{lemma:nagumo}) for forward invariance of $\mathcal{A}$ is violated at $x_0$.
	\end{example}
	

%
	
	\section{Methodology}\label{sec:methods}
	
%
	
	It is clear from Example~\ref{example} that possessing a collection of CBFs for $(\mathcal{S},\mathcal{U})$ is not sufficient to solve Problem~\ref{problem}. Thus, our first strategy is to identify other control sets $\mathcal{Y}$, for which possessing CBFs for $(\mathcal{S},\mathcal{Y})$ is sufficient to solve Problem~\ref{problem}. That is, if we restrict $w_1,w_2$ in Fig.~\ref{fig:example} to a smaller set $\mathcal{Y}\subset\mathcal{U}$ when designing our CBFs, then under certain conditions, presented in Section~\ref{sec:two_constraints}, we can ensure that several CBFs will be concurrently feasible\extendedversion{{} over the full control set $\mathcal{U}$}. 
	When this strategy fails to yield a full solution to Problem~\ref{problem}, we then present a more typical iterative algorithm in Section~\ref{sec:three_constraints} to remove the remaining infeasible states in $\mathcal{S}$. \extendedversion{We also present a brief remark on QP controllers in Section~\ref{sec:qps}.}
	
	
	\subsection{When All CBFs are Non-Interfering} \label{sec:two_constraints}
	
	Some properties we will need are as follows:
	\begin{definition} \label{def:noninterfere}
		Two CBFs $h_i$ and $h_j$ are called \emph{non-interfering on $\mathcal{X}$} if $(\nablav h_i(q,v)g_2(q)) \cdot (\nablav h_j(q,v) g_2(q)) \geq 0$ for all $(q,v)\in\mathcal{X}$. A collection of CBFs $\{h_k\}_{k\in[M]}$ is called non-interfering if every pair of CBFs is non-interfering.
	\end{definition}

	\begin{definition}\label{def:oep}
		%
		%
		Given a set $\mathcal{U}'\subset\mathcal{U}$, let $\{w_i\}_{i\in[m]}$ be a set of $m$ vectors $w_i\in\mathcal{U}'$ satisfying $w_i\cdot w_j \geq 0, \forall {i\in[m]},{j\in[m]}$. Let $\{y_i\}_{i\in[m]}$ be the set {\color{black}of orthogonal projections} $y_i = w_i - \proj_{(\{w_j\}_{j\in[i-1]})} w_i$ (or ${y_i = w_i}$ if $\{w_i\}_{i\in[m]}$ are orthogonal). The set $\mathcal{U}'$ has the \emph{orthogonal extension property (OEP) with respect to $\mathcal{U}$} if for every such set $\{w_i\}_{i\in[m]}$, the 
		point $z = \sum_{i\in[m]} y_i$ belongs to $\mathcal{U}$.
		%
	\end{definition}
	

	\begin{definition}\label{def:qep}
		Given a set $\mathcal{U}'\subset\mathcal{U}$, let $\{\mathcal{P}_i\}_{i\in[m]}$ be a set of $m$ orthogonal hyperplanes in $\reals^m$ satisfying $\mathcal{P}_i\cap\mathcal{U}' \neq \emptyset, \forall i\in[m]$. Let $p$ be the point where all $m$ hyperplanes intersect (where $p$ is guaranteed to exist because $\{\mathcal{P}_i\}_{i\in[m]}$ are orthogonal). The set $\mathcal{U}'$ has the \emph{quadrant extension property (QEP) with respect to $\mathcal{U}$} if for every such set $\{\mathcal{P}_i\}_{i\in[m]}$, the point $p$ belongs to $\mathcal{U}$.
	\end{definition}
	
	Examples of sets satisfying Definitions~\ref{def:oep}-\ref{def:qep} are as follows:
	
	\begin{example} \label{ex:oep}
		Given various prescribed input bounds $\mathcal{U}$, the following sets $\mathcal{U}'$ possess the OEP or QEP with respect to $\mathcal{U}$. Note that these choices of $\mathcal{U}'$ are not unique.
		\begin{enumerate}
			\item If $\mathcal{U} = \{ u \in \reals^m \mid \|u\|_\infty \leq \gamma\}$, then one possible set with the OEP is $\mathcal{U}' = \{u\in \reals^m \mid \|u\|_1 \leq \gamma\}$ (Fig.~\ref{fig:oep}a). One possible set with the QEP is $\mathcal{U}' = \{u\in \reals^m \mid \|u\|_1 \leq \gamma^*\}$ (Fig.~\ref{fig:oep}b), where $\gamma^*$ must be computed. If $m=2$, then $\gamma^* = \frac{2\gamma}{1+\sqrt{2}}$.
			\item If $\mathcal{U} = \{u \in \reals^m \mid \|u\|_1 \leq \gamma\}$, then one possible set with the OEP is $\mathcal{U}' = \{u\in\reals^m \mid \|u\|_\infty \leq \frac{\gamma}{m}\}$. One possible set with the QEP is $\mathcal{U}' = \{u\in\reals^m \mid \|u\|_\infty \leq \frac{\gamma^*}{m}\}$, where $\gamma^*$ is as in the prior case.
			\item If $\mathcal{U} = \{u \in \reals \mid \|u\|_2 \leq \gamma\}$, then one possible set with the OEP (Fig.~\ref{fig:oep}c) and QEP (Fig.~\ref{fig:oep}d) is $\mathcal{U}' = \{u\in\reals^m \mid \|u\|_2 \leq \frac{\gamma}{\sqrt{m}}\}$.
			\item If $\mathcal{U} = \{ u \in \reals^m \mid \max_i |a_i u_i| \leq \gamma\}$ for constants $\{a_1, \cdots, a_m\}$, then one possible set with the OEP is $\mathcal{U}' = \{u\in \reals^m \mid \sum_i |a_i u_i| \leq \gamma^*\}$ (Fig.~\ref{fig:oep}e), where $\gamma^* \leq \gamma$ must be computed. A set with the QEP can be constructed similarly (Fig.~\ref{fig:oep}f).
		\end{enumerate}  
	\end{example}

	\begin{figure}
		\centering
		\begin{tikzpicture}
			\node[anchor=north,inner sep=0] (image) at (0,0) 	{\includegraphics[width=0.4\columnwidth,trim={0.75in, 0.045in, 1in, 0.3in},clip]{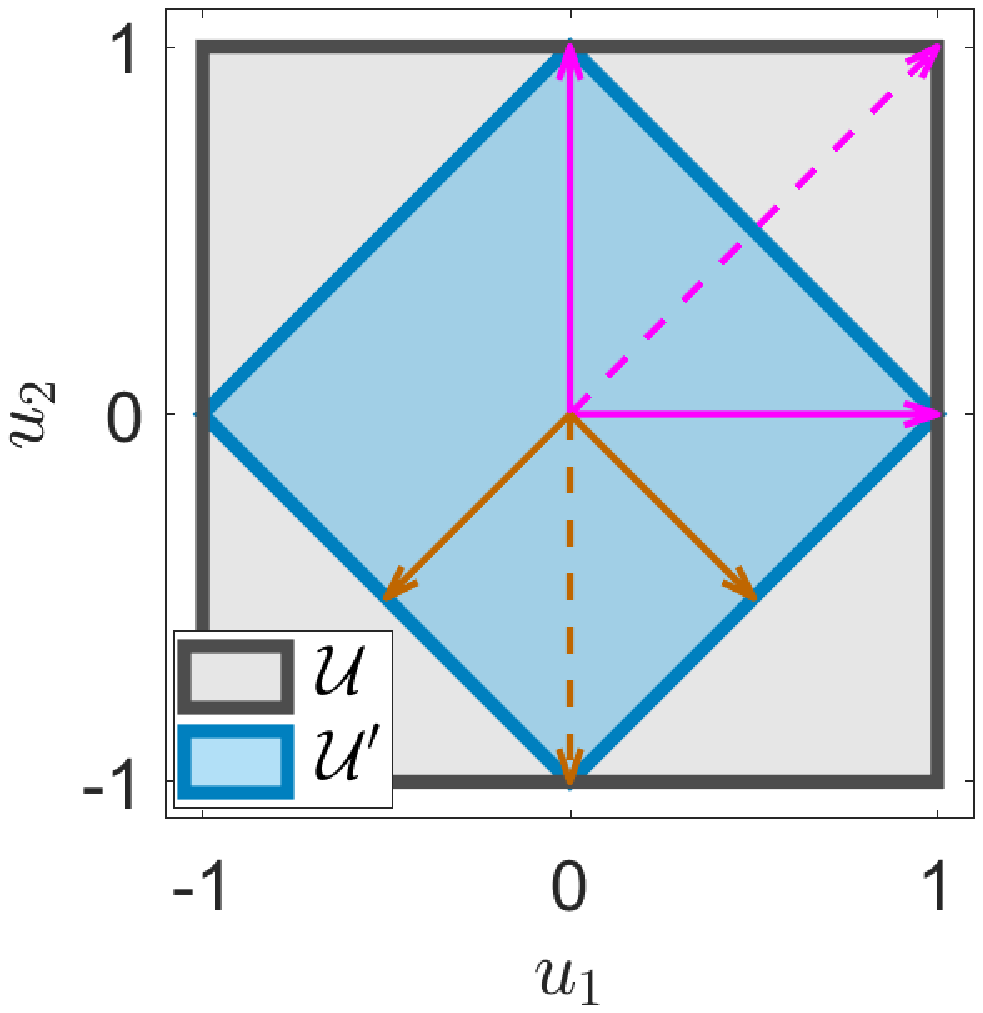}};
			\node[anchor=north,inner sep=0] (image) at (0.42\columnwidth,0) {\includegraphics[width=0.4\columnwidth,trim={0.75in, 0.045in, 1in, 0.3in},clip]{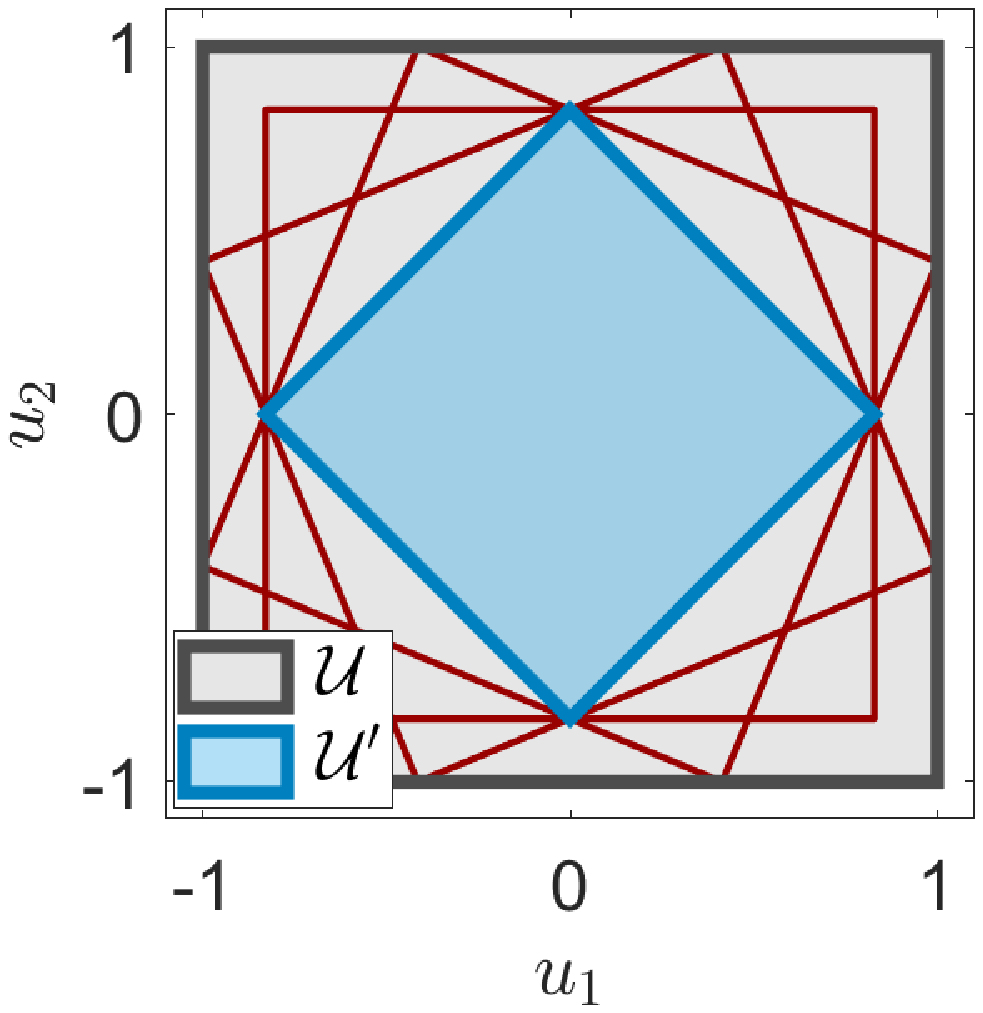}};
			\node[anchor=north,inner sep=0] (image) at (0,-1.4in) {\includegraphics[width=0.4\columnwidth,trim={0.75in, 0.045in, 1in, 0.3in},clip]{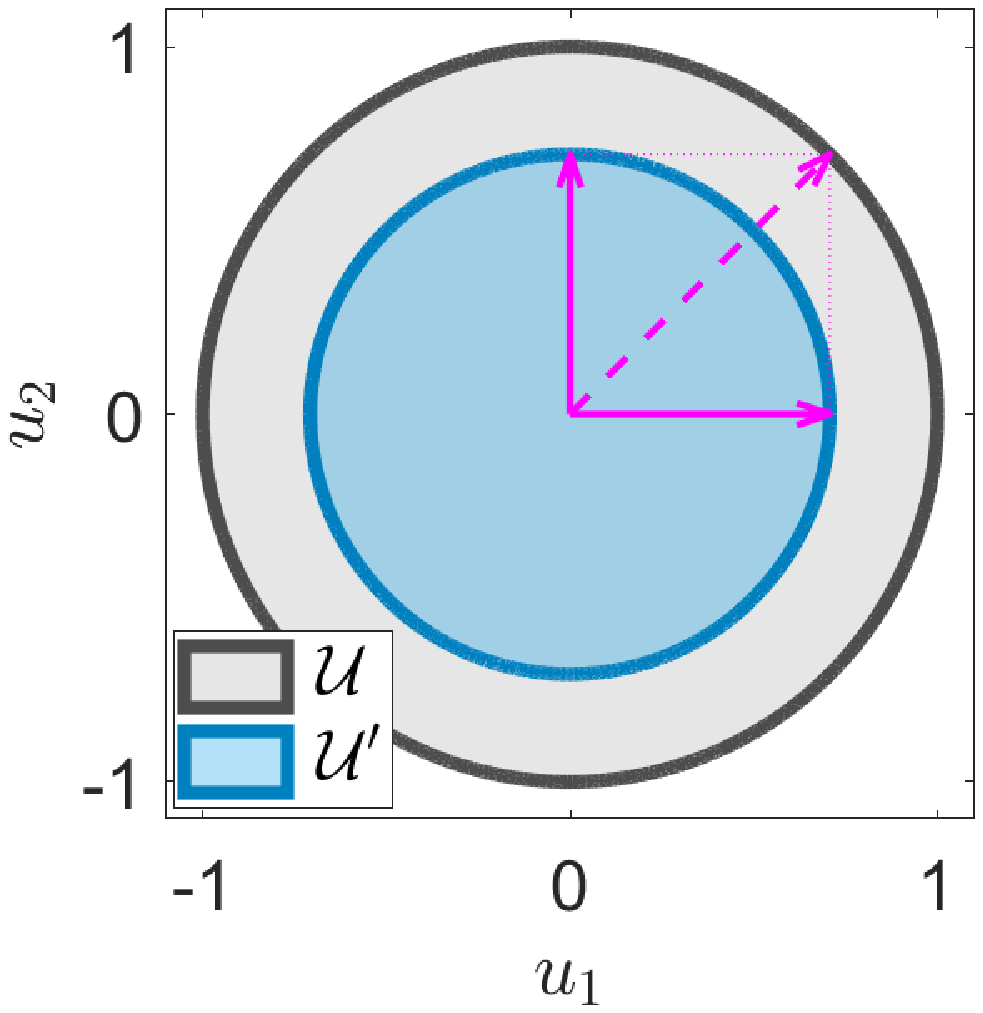}};
			\node[anchor=north,inner sep=0] (image) at (0.42\columnwidth,-1.4in) {\includegraphics[width=0.4\columnwidth,trim={0.65in, 0.045in, 1in, 0.3in},clip]{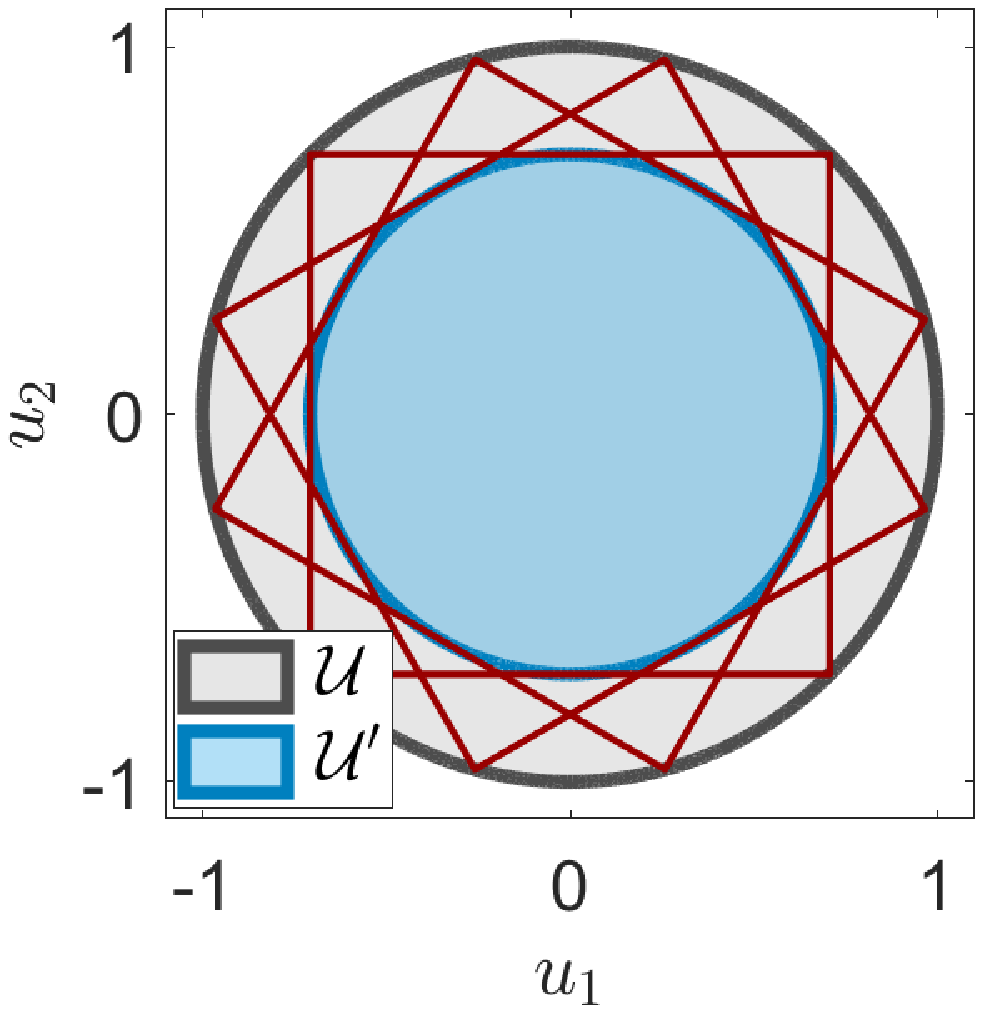}};
			\node[anchor=north,inner sep=0] (image) at (-0.06\columnwidth,-2.8in) {\includegraphics[width=0.49\columnwidth,trim={0.1in, 0.3in, 0.5in, 0.95in},clip]{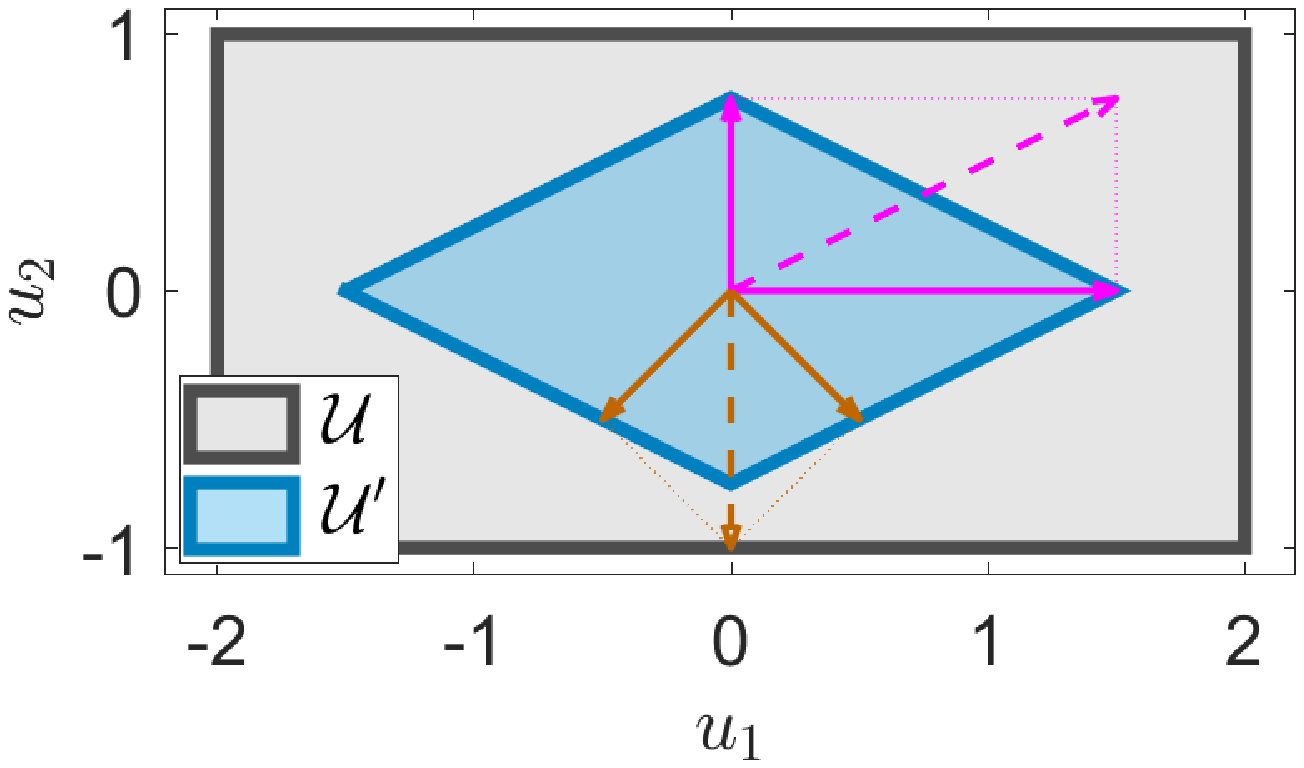}};
			\node[anchor=north,inner sep=0] (image) at (.455\columnwidth,-2.8in) {\includegraphics[width=0.49\columnwidth,trim={0.1in, 0.3in, 0.5in, 0.95in},clip]{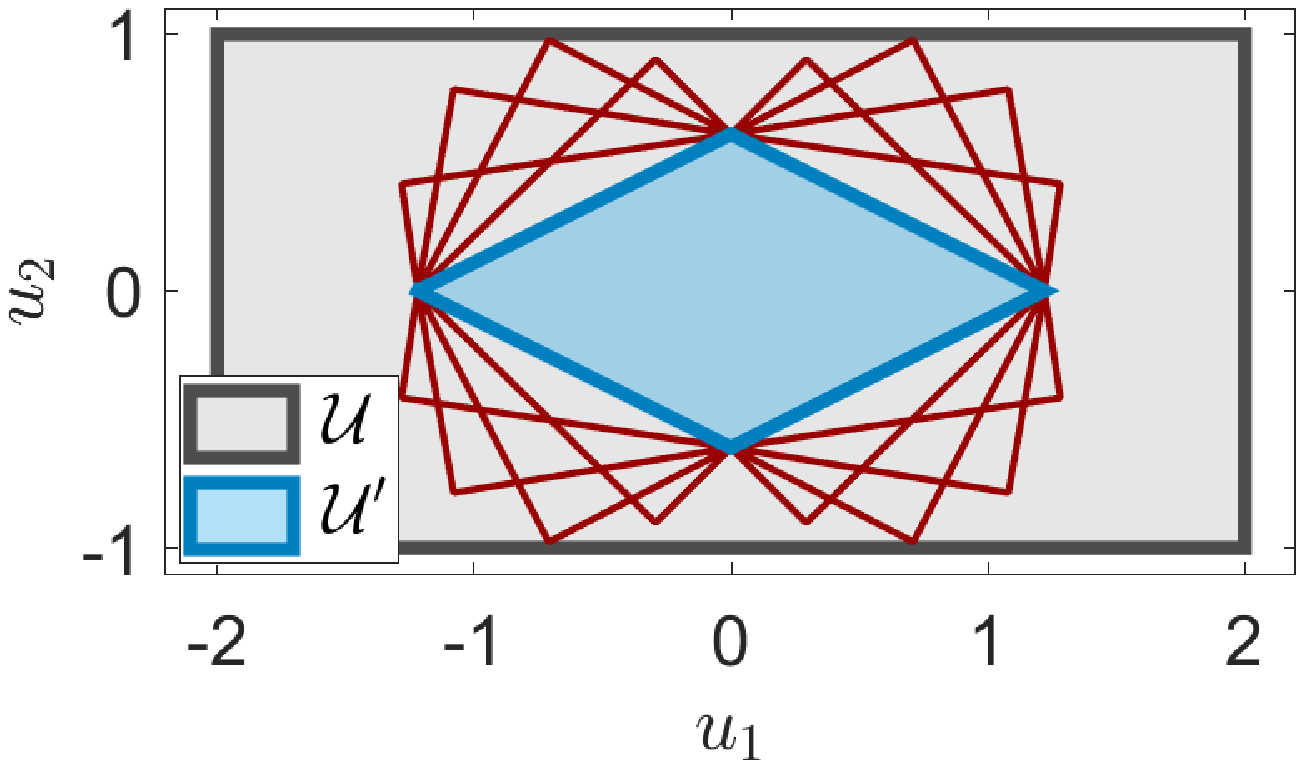}};
			\node [anchor=north] (note) at (-1.75,0.0in) {(a)}; 
			\node [anchor=north] (note) at (-1.75,-1.4in) {(c)}; 
			\node [anchor=north] (note) at (-1.65,-2.84in) {(e)}; 
			\node [anchor=north] (note) at (5.6,0.0in) {(b)}; 
			\node [anchor=north] (note) at (5.6,-1.4in) {(d)}; 
			\node [anchor=north] (note) at (5.6,-2.84in) {(f)};
		\end{tikzpicture}
		
		\regularversion{\vspace{-4pt}}\caption{Given three different control sets $\mathcal{U}$ (gray), the above illustrates possible choices of set $\mathcal{U}'$ (blue) that: left) have the OEP with respect to $\mathcal{U}$, and right) have the QEP with respect to $\mathcal{U}$, as in Definitions~\ref{def:oep}-\ref{def:qep}. The left plots also show how given two orthogonal vectors (solid arrows) in $\mathcal{U}'$, their sum (dashed arrow) must by construction belong to $\mathcal{U}$. The right plots also show various choices of orthogonal hyperplanes (i.e. lines in $\reals^2$) whose intersections (the red right angles) by construction must belong to $\mathcal{U}$.}
		\label{fig:oep}
		\regularversion{\vspace{-14pt}}\extendedversion{\vspace{-6pt}}
	\end{figure}

	The core idea of this subsection is that given a set $\mathcal{U}'$ with the OEP or QEP, if we design our CBFs $\{h_k\}_{k\in[M]}$ one-at-a-time for $(\mathcal{S},\mathcal{U}')$, then, subject to an additional condition on the CBF gradients, we can guarantee a priori that $\{h_k\}_{k\in[M]}$ will have jointly feasible CBF conditions. To show this, we begin with several lemmas about the geometry of the OEP and QEP, first for two constraints, and then for $M$ constraints.
	
	
	\regularversion{\vspace{-2pt}}
	\begin{lemma} \label{lemma:less_than_90}
		Let $\mathcal{U}'$ have the OEP with respect to $\mathcal{U}$. Given two row vectors $A_1,A_2\in\reals^{1\times m}$ and two scalars $b_1,b_2\in\reals$, if $A_1 \cdot A_2 \geq 0$ and there exists $w_1,w_2 \in \mathcal{U}'$ such that 1) $A_1 w_1 \leq b_1$, 2) $A_2 w_2 \leq b_2$, 3) $A_1\cdot w_1 = -\|A_1\|_2\|w_1\|_2$, and 4) $A_2\cdot w_2 = -\|A_2\|_2\|w_2\|_2$, then there exists $z\in\mathcal{U}$ such that $A_1 z \leq b_1$ and $A_2 z \leq b_2$.
	\end{lemma}
	\begin{proof}
		Note that conditions 3 and 4 imply that either 1) $w_1$ is parallel and opposite to $A_1$ when $b_1 < 0$, or 2) that $w_1 = 0$ when $b_1 \geq 0$, and similarly for $A_2$, $b_2$, $w_2$.
		Without loss of generality, assume that $\|w_1\|_2\geq\|w_2\|_2$. Let $w^* = w_2 - \frac{w_2 \cdot w_1}{w_1 \cdot w_1}w_1$. Then $z = w_1 + w^*$ satisfies $A_1 z = A_1 w_1 \leq b_1$ since $w^*$ is orthogonal to $w_1$ and $A_1$, and 
		\vspace{-2pt}$$ A_2 z = A_2 w_2 + \underbrace{A_2 w_1}_{\leq 0} (1 - \underbrace{\textstyle\frac{w_2\cdot w_1}{w_1\cdot w_1}}_{\leq 1})\leq A_2 w_2 \leq b_2 \,.\vspace{-6pt}$$
		By the OEP, $z \in\mathcal{U}$.
	\end{proof}

	\begin{lemma}\label{lemma:oep_all}
		Let $\mathcal{U}'$ have the OEP with respect to $\mathcal{U}$. Given $M$ row vectors $\{A_k\}_{k\in[M]}$ and scalars $\{b_k\}_{k\in[M]}$ with $A_k\in\reals^{1\times m}$, if 1) $A_i\cdot A_j \geq 0$ for all $i\in[M],j\in[M]$, 2) there exists $w_k\in\mathcal{U}'$ such that $A_k w_k \leq b_k$ for all $k\in[M]$, and 3) $A_k w_k = -\|A_k\|_2 \|w_k\|_2$ for all $k\in[M]$, then there exists $z\in\mathcal{U}$ such that $A_k z \leq b_k$ for all $k\in[M]$.
	\end{lemma}
	\regularversion{\vspace{-10pt}}
	\begin{proof}
		The proof follows from that of Lemma~\ref{lemma:less_than_90}. Assume that the vectors are ordered by decreasing $\|w_k\|_2$. If $M \leq m$, then use orthogonal projections to construct a vector $z = \sum_{i\in[M]} (w_i - \proj_{(\{w_j\}_{j\in[i-1]})}w_i)$ that satisfies all $M$ constraints $A_k z \leq b_k$ simultaneously. By the OEP, $z\in\mathcal{U}$. If $M > m$, then because $A_i \cdot A_j \geq 0$ for all $i,j$, the set $\mathcal{W} = \{u\in\reals^m\mid A_k u \leq b_k, \forall k\in[M]\}$ must contain a complete orthant $\mathbb{O}=\{u\in\reals^m \mid O_k u \leq O_k y, \forall k\in[m]\}\subseteq\mathcal{W}$ for some orthogonal set $\{O_k\}_{k\in[m]}$, $O_k\in\reals^{1\times m}$ and some point $y\in\reals^m$. 
		Therefore, at least $l_0 = M - m$ of the constraints $A_k z \leq b_k$ must be redundant, i.e. $l\geq l_0$ of the constraints must be automatically satisfied if the remaining constraints are satisfied. Use the remaining $M-l$ constraints as in the prior case to construct a vector $z\in\mathcal{U}$ satisfying all $M$ inequalities simultaneously.\regularversion{\vspace{-3pt}}
	\end{proof}

	\begin{lemma} \label{lemma:qep}
		Let $\mathcal{U}'$ have the QEP with respect to $\mathcal{U}$. Given two row vectors $A_1,A_2\in\reals^{1\times m}$ and two scalars $b_1,b_2\in\reals$, if $A_1 \cdot A_2 \geq 0$ and there exists $w_1, w_2 \in \mathcal{U}'$ such that $A_1 w_1 \leq b_1$ and $A_2 w_2 \leq b_2$, then there exists $p\in\mathcal{U}$ such that $A_1 p \leq b_1$ and $A_2 p \leq b_2$.
	\end{lemma}
	\regularversion{\vspace{-10pt}}\begin{proof}
		Consider the following figure:\extendedversion{\vspace{1.5in}} \\
		\begin{minipage}{\columnwidth}
		\centering
	\resizebox{0.9\columnwidth}{!}{%
		\begin{tikzpicture}
			\node[anchor=south east,inner sep=0] (image) at (0,0) {\includegraphics[width=\columnwidth,clip,trim={0.35in, 0.2in, 0.45in, 0.5in}]{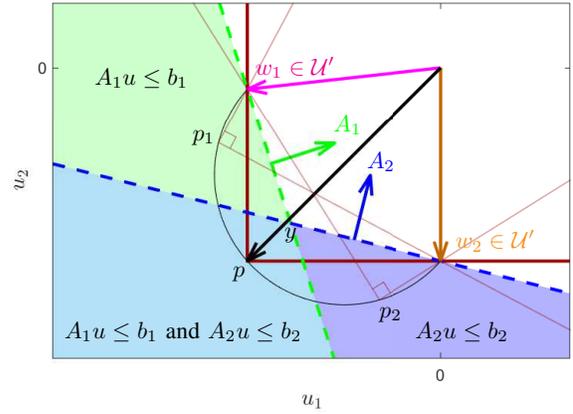}}; 
			\node [anchor=south east] (note) at (-3.2,4.1) {\color{green}$A_1$}; 
			\node [anchor=south east] (note) at (-2.7,3.55) {\color{blue}$A_2$};
			\node [anchor=south east] (note) at (-3.6,4.95) {\color{magenta}$w_1\in\mathcal{U}'$};
			\node [anchor=south east] (note) at (-0.6,2.35) {\color{BurntOrange}$w_2\in\mathcal{U}'$};
			\node [anchor=south east] (note) at (-4.95,1.9) {$p$};
			\node [anchor=south east] (note) at (-5.4,3.95) {$p_1$};
			\node [anchor=south east] (note) at (-2.6,1.3) {$p_2$};
			\node [anchor=south east] (note) at (-5.8,4.8) {$A_1 u \leq b_1$};
			\node [anchor=south east] (note) at (-1.0,1.0) {$A_2 u \leq b_2$};
			\node [anchor=south east] (note) at (-4.1,1) {$A_1 u \leq b_1$ and $A_2 u \leq b_2$};
			\node [anchor=south east] (note) at (-4.17,2.48) {$y$};
		\end{tikzpicture}
	}%
		\regularversion{\vspace{-5pt}}
		\captionof{figure}{Visualization for Lemma~\ref{lemma:qep} proof.}
		\label{fig:qep_proof}
		\vspace{6pt}
		\end{minipage}
		Since $w_1,w_2\in\mathcal{U}'$, by the QEP, any two orthogonal lines (i.e. hyperplanes in $\reals^2$) that intersect $w_1$ and $w_2$ (i.e. intersect points in $\mathcal{U}'$) must meet at a point in $\mathcal{U}$, such as the points $p,p_1,p_2$ above. That is, every point on the black arc must belong to $\mathcal{U}$. Next, since $A_1 \cdot A_2$ is at least zero, the point $y$ where the hyperplanes $\{u \mid A_1 u = b_1\}$ and $\{u \mid A_2 u = b_2\}$ (dashed lines) intersect must be enclosed by the black arc. It follows that at least one point, above labeled $p$, on this arc must satisfy both inequalities simultaneously.\regularversion{\vspace{-2pt}}
	\end{proof}

	\begin{lemma} \label{lemma:qep_all}
		Let $\mathcal{U}'$ have the QEP with respect to $\mathcal{U}$. Given $M$ row vectors $\{A_k\}_{k\in[M]}$ and scalars $\{b_k\}_{k\in[M]}$ with $A_k\in\reals^{1\times m}$, if $A_i \cdot A_j \geq 0$ for all $i\in[M],j\in[M]$ and there exists $w_k \in \mathcal{U}'$ such that $A_k w_k \leq b_k$ for all $k\in[M]$, then there exists $p\in\mathcal{U}$ such that $A_k p \leq b_k$ for all $k\in[M]$.
	\end{lemma}
	\regularversion{\vspace{-10pt}}\begin{proof}
		The argument is similar to that in Lemma~\ref{lemma:qep}, now extended to higher dimension. Let $\{\mathcal{P}_k\}_{k\in[m]}$ be a set of orthogonal hyperplanes in $\reals^m$ that meet at some point $p\in\reals^m$ and satisfy $\mathcal{P}_k\cap\mathcal{U}'\neq\emptyset,\forall k\in[m]$. Let $\mathbb{P} = \{ u\in\reals^m \mid P_k u \leq P_k p, \forall k\in[m]\}$ be the orthant of $\reals^m$ originating from $p$ and enclosed by $\{\mathcal{P}_k\}_{k\in[m]}$, for appropriate vectors $\{P_k\}_{k\in[m]},P_k\in\reals^{1\times m}$. 
		As in Lemma~\ref{lemma:oep_all}, there are at most $m$ non-redundant constraints. Let $\mathcal{N}$ be the set of indices of these non-redundant constraints, and construct $\mathbb{P}$ so that $w_k \in \partial \mathbb{P}, \forall k \in \mathcal{N}$, analogous to how the solid red lines ($\partial\mathbb{P}$) intersect the vectors $w_1,w_2$ in Fig.~\ref{fig:qep_proof}. Then the point $p$ will lie on the boundary of an $m$-hypersphere $\mathbb{S}$ analogous to the black arc in Fig.~\ref{fig:qep_proof}, and all possible choices of $p$ must lie in $\mathcal{U}$ by the QEP. Let $y\in\reals^m$ satisfy $A_k y = b_k, \forall k \in \mathcal{N}$. By the same argument as in Lemma~\ref{lemma:qep}, $\mathbb{S}$ must enclose at least one such $y$. Thus, there exists at least one $p\in\mathbb{S}\subseteq\mathcal{U}$ that satisfies all $M$ inequalities simultaneously. 
	\end{proof}
	\regularversion{\vspace{-3pt}}

	We now apply the above geometric observations to the concurrent feasibility of several SCBFs or CBFs as follows.\regularversion{\vspace{-2pt}}

	\regularversion{\vspace{-1pt}}\begin{theorem} \label{thm:less_than_90}
		Let $\mathcal{U}'$ be any set with the OEP with respect to $\mathcal{U}$. Let $\{h_k\}_{k\in[M]}$ each be an SCBF for $(\mathcal{S},\mathcal{U}')$. If $\{h_k\}_{k\in[M]}$ are non-interfering on $\mathcal{S}$ as in Definition~\ref{def:noninterfere}, then the set $\boldsymbol\mu_\textrm{all}$ in \eqref{eq:cbf_condition} is nonempty for all $(q,v)\in\mathcal{S}\cap(\cap_{k\in[M]}\mathcal{H}_k)$.
	\end{theorem}
	\regularversion{\vspace{-10pt}}\begin{proof}
		Let $A_k = \nablav h_k(q,v) g_2(q)$ and $b_k = \alpha_k(-h_k(q,v))$ $ - \nablaq h_k(q,v)g_1(q) v$ for all $k\in[M]$, where $\alpha_k$ each come from \eqref{eq:control_set}. It follows from Definition~\ref{def:noninterfere} that $A_i \cdot A_j \geq 0$ for all ${i\in[M],j\in[M]}$ everywhere in $\mathcal{S}$. It follows from Definition~\ref{def:scbf} that for each ${k\in[M]}$, there exists $w_k\in\mathcal{U}'$ such that $A_k w_k \leq b_k$ and $w_k$ is {\color{black}anti-}parallel to $A_k$. Lemma~\ref{lemma:oep_all} then implies that these $M$ inequalities are simultaneously feasible for some $u$ in the full set $\mathcal{U}$, which is equivalent to the sets $\{\boldsymbol\mu_k\}_{k\in[M]}$ having a nonempty intersection $\boldsymbol\mu_\textrm{all}$.\regularversion{\vspace{-2pt}}
		%
	\end{proof}

	\begin{theorem} \label{thm:less_than_90_qep}
		Let $\mathcal{U}'$ be any set with the QEP with respect to $\mathcal{U}$. Let $\{h_k\}_{k\in[M]}$ each be a CBF for $(\mathcal{S},\mathcal{U}')$. If $\{h_k\}_{k\in[M]}$ are non-interfering on $\mathcal{S}$ as in Definition~\ref{def:noninterfere}, then the set $\boldsymbol\mu_\textrm{all}$ in \eqref{eq:cbf_condition} is nonempty for all $(q,v)\in\mathcal{S}\cap(\cap_{k\in[M]}\mathcal{H}_k)$.
	\end{theorem}
	\regularversion{\vspace{-10pt}}\begin{proof}
		The proof is identical to that of Theorem~\ref{thm:less_than_90}, except that $h_k$ are CBFs instead of SCBFs, so each $w_k$ is not guaranteed to be {\color{black}anti-}parallel to each $A_k$, respectively. Thus, we apply Lemma~\ref{lemma:qep_all} instead of Lemma~\ref{lemma:oep_all}.\regularversion{\vspace{-2pt}}
	\end{proof}

	Theorems~\ref{thm:less_than_90}-\ref{thm:less_than_90_qep} constitute our first method for ensuring concurrent feasibility of multiple CBFs.
	Note that we provide theorems under both the OEP and QEP separately, because as shown in Fig.~\ref{fig:oep}, the OEP is less conservative (i.e. allows for larger $\mathcal{U}'$), but is only applicable when the stricter SCBF condition \eqref{eq:scbf_control_set} holds (which is often the case in practice).
	%
	We also note the following remark about the computation of the gradients of the CBFs $h_k$ for Definition~\ref{def:noninterfere} and Theorems~\ref{thm:less_than_90}-\ref{thm:less_than_90_qep}.\regularversion{\vspace{-2pt}}
	
	\regularversion{\vspace{-1pt}}\begin{remark} \label{remark:high_degree}
		If $\kappa$ is \regularversion{an HOCBF}\extendedversion{a High Order CBF} as in \cite{Recursive_CBF}, then there exists a function $\phi\in\mathcal{K}$ such that $h(q,v) = \dot{\kappa}(q,v) - \phi(-\kappa(q))$ is a CBF as in Definition~\ref{def:cbf}. Moreover, $\dot{\kappa}(q,v) = \nablaq \kappa(q) g_1(q) v$, so $\nablav h(q,v) \equiv \nablaq \kappa(q) g_1(q) g_2(q)$. That is, $\nablav h(q,v)$ 1) is independent of $v$ and 2) does not depend on the function $\phi$. Thus, 
		Definition~\ref{def:noninterfere}
		can be checked using only the gradients of the constraint functions $\kappa_i,\kappa_j$ and the dynamics \eqref{eq:model} without knowing the exact CBFs (i.e. the choices of $\phi$).\regularversion{\vspace{-2pt}}
	\end{remark}
	\regularversion{\vspace{-1pt}}
	
	Referring to Example~\ref{example}, Theorems~\ref{thm:less_than_90}-\ref{thm:less_than_90_qep} address the problem of determining a viability domain $\mathcal{A}$ when $\gamma \in (0,1]$. However, how to address Example~\ref{example} when $\gamma > 1$ still needs to be determined, as is done in Section~\ref{sec:three_constraints}.
%

	\begin{example}\label{ex:fixed_less_than_90}
		Consider the problem in Example~\ref{example} with $\gamma = 0.75$. The set $\mathcal{U}' = \{ u \in \reals^2 \mid \| u\|_1 \leq \frac{2}{1+\sqrt{2}}\}$ has the QEP with respect to $\mathcal{U}$ as in Example~\ref{example}. Using the new set $\mathcal{U}'$ as the assumed allowable input bounds, the method in \cite{Recursive_CBF,Automatica} yields the CBFs $h_i(q,v) = \dot{\kappa}_i(q) - \sqrt{-\frac{4}{1+\sqrt{2}}\kappa_i(q)}$. Unlike in Example~\ref{example}, the new set $\mathcal{A} = \mathcal{S}\cap\mathcal{H}_1\cap\mathcal{H}_2$ is indeed a viability domain. The impact of constructing the CBFs for $(\mathcal{S},\mathcal{U}')$ instead of $(\mathcal{S},\mathcal{U})$ is visualized in the left half of Fig.~\ref{fig:fixes}.
	\end{example}
	\regularversion{\vspace{-6pt}}


	\begin{figure}
		\centering
		\begin{tikzpicture}
			\node[anchor=south west,inner sep=0] (image) at (0,0) {\includegraphics[width=0.49\columnwidth,clip,trim={0.5in, 0.03in, 0.8in, 0.3in}]{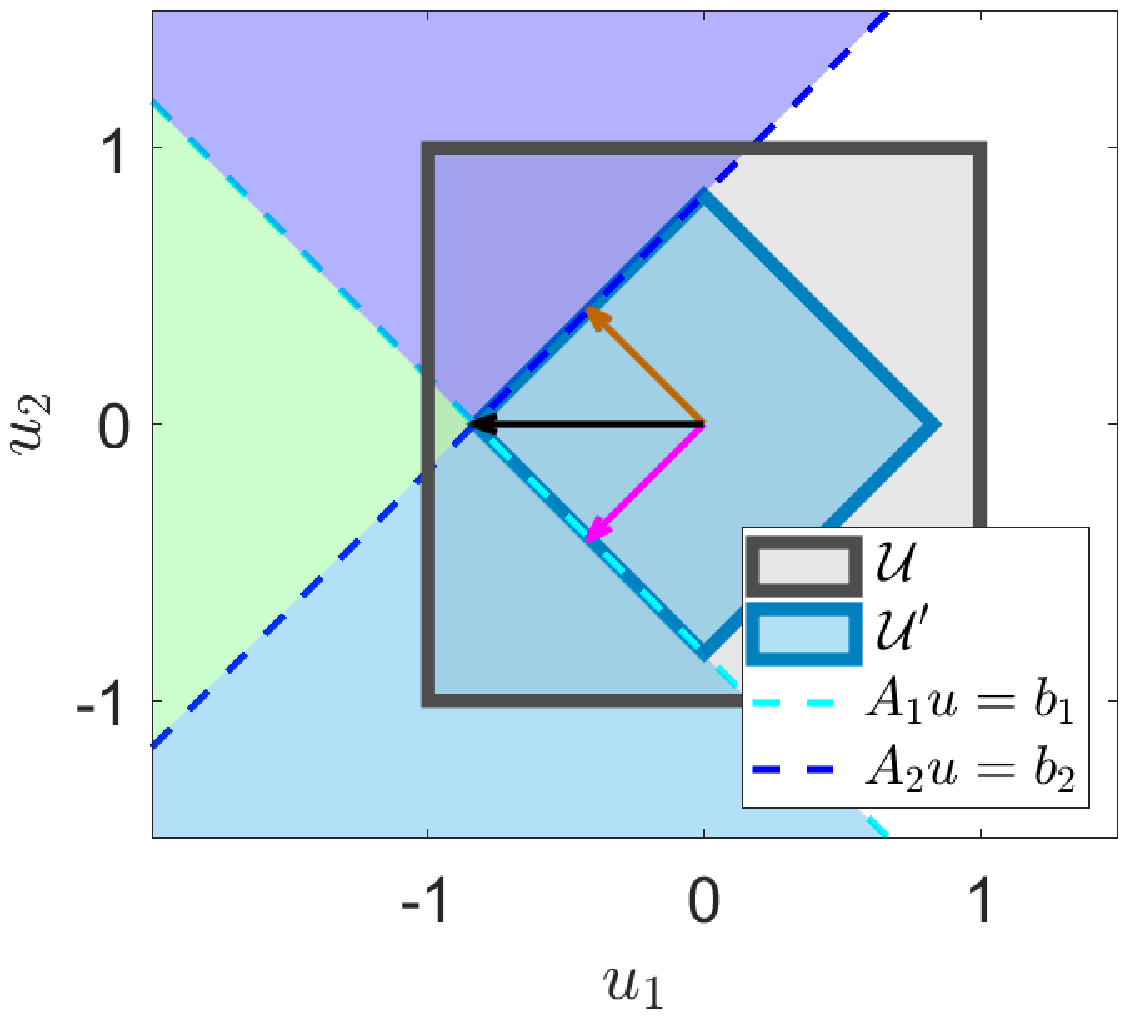}};
			\node[anchor=south east,inner sep=0] (image) at (\columnwidth,0) {\includegraphics[width=0.49\columnwidth,clip,trim={0.5in, 0.03in, 0.8in, 0.3in}]{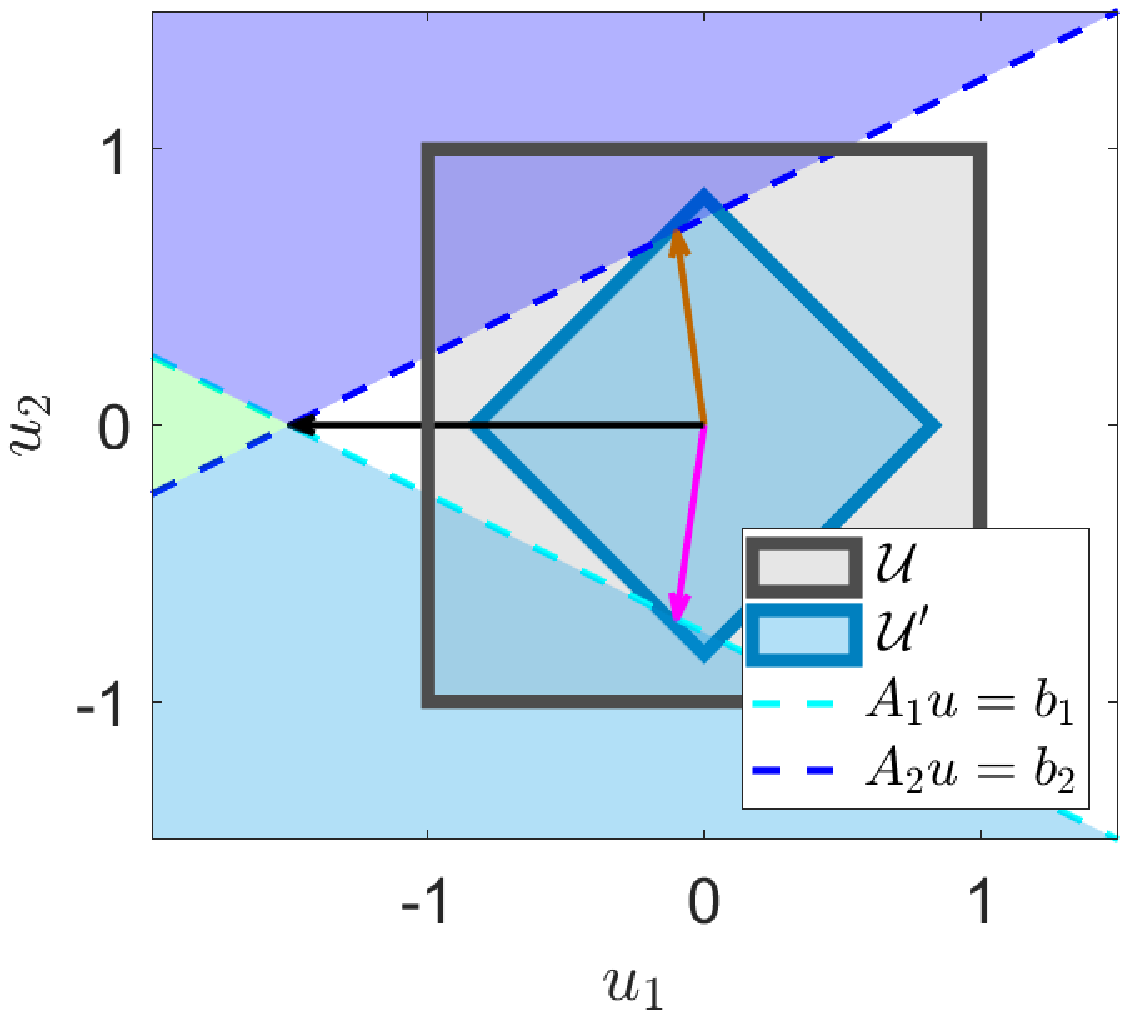}};
			\node [anchor=south west] (note) at (2.2,1.5) {\color{magenta}$w_1$}; 
			\node [anchor=south west] (note) at (2.2,2.3) {\color{BurntOrange}$w_2$}; 
			\node [anchor=south west] (note) at (1.55,2.2) {$z$}; 
			\node [anchor=south west] (note) at (0.8,3.25) {$A_2 u \leq b_2$}; 
			\node [anchor=south west] (note) at (0.8,0.58) {$A_1 u \leq b_1$};  
			%
			\node [anchor=south west] (note) at (6.42,1.62) {\color{magenta}$w_1$}; 
			\node [anchor=south west] (note) at (6.85,2.3) {\color{BurntOrange}$w_2$}; 
			\node [anchor=south west] (note) at (5.25,2.2) {$z$}; 
			\node [anchor=south west] (note) at (4.9,3.25) {$A_2 u \leq b_2$}; 
			\node [anchor=south west] (note) at (4.9,0.58) {$A_1 u \leq b_1$}; 
		\end{tikzpicture}
		\caption{Left: Visualization of the control space at a point $x_0\in\partial\mathcal{H}_1\cap\partial\mathcal{H}_2$ using $\mathcal{H}_1,\mathcal{H}_2$ as in Example~\ref{ex:fixed_less_than_90}. Compared to Fig.~\ref{fig:example}, the CBFs are now designed over $\mathcal{U}'$ (the blue diamond) instead of $\mathcal{U}$, so there exists $z\in\mathcal{U}$ satisfying both $A_1 z\leq b_1$ and $A_2 z \leq b_2$. Right: Visualization of the control space at the point $x_0$ in Example~\ref{ex:bad_greater_than_90}. Because the CBFs $h_1,h_2$ do not satisfy Definition~\ref{def:noninterfere}, it may be the case that every point $z$ satisfying both inequalities $A_k z \leq b_k$ lies outside the set $\mathcal{U}$, so $x_0$ should not lie in the viability domain.}
		\label{fig:fixes}
		\regularversion{\vspace{-10pt}}\extendedversion{\vspace{-6pt}}
	\end{figure}
	
	\subsection{Modifying the Quadratic Program} \label{sec:qps}
	
	The work in the prior section required Definition~\ref{def:noninterfere} to apply to all states in $\mathcal{S}$. Now, recall that Lemma~\ref{lemma:nagumo} is effectively only a condition on the boundary of the invariant set $\mathcal{A}$. Noting this, in Section~\ref{sec:three_constraints}, we will focus only on the boundary of $\mathcal{A}$, so unlike in Section~\ref{sec:two_constraints} we will no longer be able to guarantee that $\boldsymbol\mu_\textrm{all}(q,v)$ is nonempty for $(q,v)\in\interior(\mathcal{A})$. Rather, we will only be able to guarantee that there exists $u\color{black}\in\mathcal{U}\color{black}$ such that $\dot{h}_k(q,v,u) \leq 0$ for all $k$ in the \emph{active set}
	\begin{equation}
		\mathcal{I}(q,v) = \{k\in[M] \mid h_k(q,v) = 0 \} \label{eq:active_set} \,.
	\end{equation}
	{\color{black}That is, after employing the algorithm in Section~\ref{sec:three_constraints}, the condition $\dot{h}_k(q,v,u) \leq \alpha_k(-h_k(q,v))$ in \eqref{eq:first_qp_conditions} will be feasible for all $k\in\mathcal{I}(q,v)$, but possibly not for $k\in[M]\setminus\mathcal{I}(q,v)$. By \cite{CBFs_Necessary}, if $\mathcal{A}$ in Lemma~\ref{lemma:all_invariance} is a viability domain, then there exists a set of class-$\mathcal{K}$ functions $\{\alpha_k^*\}_{k\in[M]}$, such that \eqref{eq:first_qp} is always feasible. Instead of computing such a set 
	directly, we let $\alpha_k^*(\lambda) = \delta_k \alpha_k(\lambda)$, where $\delta_k$ is a free variable. We then}
	let $J_k > 0$, and modify the QP \eqref{eq:first_qp} as in \cite{Choice_of_class_k} to
	\begin{subequations}
		\begin{align}
			\hspace{-4pt}&u(q,v) = \argmin_{u\in\mathcal{U},\delta_k \geq 1} \|u - u_\textrm{nom}(q,v)\|_2^2 + \sum_{k\in[M]} J_k \delta_k \label{eq:the_qp_norm} \hspace{-3pt}\\
			&\;\;\;\;\;\;\;\;\textrm{ s.t. } \dot{h}_k(q,v,u) \leq \delta_k\alpha_k(-h_k(q,v)), \forall k\in[M] \hspace{-2pt} \label{eq:the_qp_conditions}
		\end{align}\label{eq:the_qp}%
	\end{subequations}%
	\vspace{-12pt}

	\begin{proposition} \label{prop:qp}
		Suppose the conditions of Lemma~\ref{lemma:all_invariance} hold. Then 1) the control law \eqref{eq:the_qp} will render $\mathcal{A}$ forward invariant for as long as \eqref{eq:the_qp} is feasible, 2) \eqref{eq:the_qp} is feasible for all $(q,v)\in\mathcal{A}$ for which $\mathcal{I}(q,v)$ has cardinality of 0 or 1, and 3) \eqref{eq:the_qp} is feasible for all $(q,v)\in\mathcal{A}$ if $\mathcal{A}$ is a viability domain.
	\end{proposition}

	
	
	\subsection{When the CBFs are Interfering} \label{sec:three_constraints}
	
	Next, consider what happens in Example~\ref{ex:fixed_less_than_90} if $\gamma > 1$.
	
	\begin{example} \label{ex:bad_greater_than_90}
		Consider the problem in Example~\ref{ex:fixed_less_than_90} for $\gamma = 1.25$. Using the same strategy as in Example~\ref{ex:fixed_less_than_90} yields the new CBFs $h_i(q,v) = \dot{\kappa}_i(q,v) - \sqrt{-\frac{4\gamma}{1+\sqrt{2}}\kappa_i(q)}$. Denote $\mathcal{A} = \mathcal{S}\cap \mathcal{H}_1\cap\mathcal{H}_2$ and let $x_0 = (q,v) = (-\frac{1+\sqrt{2}}{4\gamma},0,1,0)\in\partial\mathcal{A}$. 
		From the right side of Fig.~\ref{fig:fixes}, we see that there is no $u\in\mathcal{U}$ that satisfies both CBF conditions at $x_0$. That is, Nagumo's necessary condition for forward invariance of $\mathcal{A}$ is violated at $x_0$, so $\mathcal{A}$ is not a viability domain. 
	\end{example}
	\regularversion{\vspace{-2pt}}

	The difference between Example~\ref{ex:fixed_less_than_90} and Example~\ref{ex:bad_greater_than_90} is that in Example~\ref{ex:bad_greater_than_90}, the CBFs $h_1,h_2$ do not satisfy Definition~\ref{def:noninterfere}, so Theorems~\ref{thm:less_than_90}-\ref{thm:less_than_90_qep} do not apply. Thus, we need additional tools to systematically remove states such as $x_0$ in Example~\ref{ex:bad_greater_than_90} from $\mathcal{A}$ in such cases. We begin by presenting a solution to this simple example before discussing a general algorithm.

	\begin{example}\label{ex:fixed_greater_than_90}
		Consider the problem in Example~\ref{ex:bad_greater_than_90}. Introduce $\kappa_3(q) = q_1$, resulting in the constraint geometry in Fig.~\ref{fig:three_constraints}. 
		Let $\mathcal{H}_1$ and $\mathcal{H}_2$ be as in Example~\ref{ex:bad_greater_than_90}, and using \cite{Recursive_CBF,Automatica}, we can derive the CBF $h_3(q,v) = \dot{\kappa}_3(q,v) - \sqrt{-\frac{4}{1+\sqrt{2}}\kappa_3(q)}$. Then $\mathcal{A} = \mathcal{S}\cap\mathcal{H}_1 \cap \mathcal{H}_2 \cap \mathcal{H}_3$ is a viability domain for the sets $\mathcal{S}$ and $\mathcal{U}$ (see the green set on the right side of Fig.~\ref{fig:three_constraints}).
	\end{example}
	\regularversion{\vspace{-2pt}}

	\begin{figure}
		\centering
	\resizebox{0.53\columnwidth}{!}{%
		\begin{tikzpicture}
			\node[anchor=south west,inner sep=0] (image) at (0,0) {\includegraphics[width=0.6\columnwidth,clip,trim={0.85in, 0.15in, 0.9in, 0.3in}]{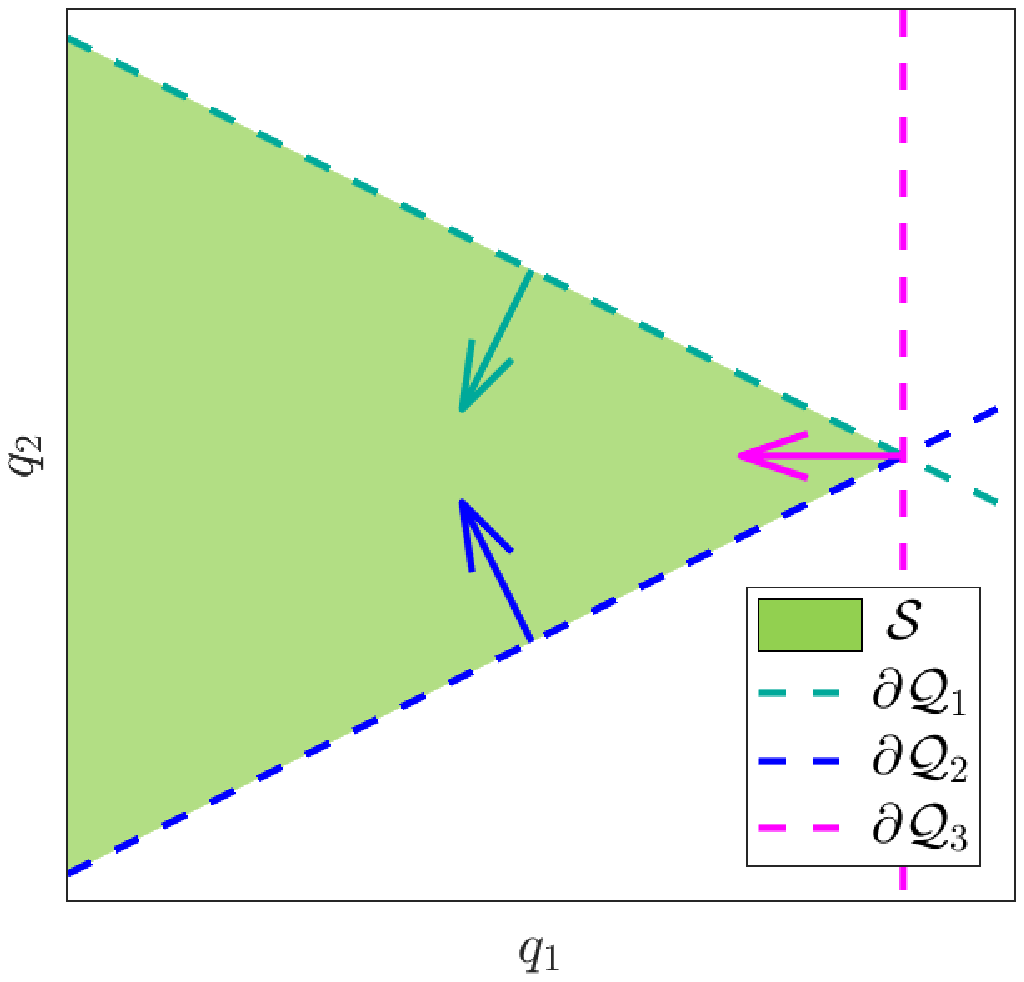}}; 
			\node [anchor=south west] (note) at (1.4,2.65) {\color{JungleGreen}$\nablaq h_1$}; 
			\node [anchor=south west] (note) at (1.4,2.05) {\color{blue}$\nablaq h_2$};
			\node [anchor=south west] (note) at (2.8,2.35) {\color{magenta}$\nablaq h_3$};
		\end{tikzpicture}
	}
	\hspace{-6pt}
	\resizebox{0.46\columnwidth}{!}{%
		\begin{tikzpicture}
			\node[anchor=south west,inner sep=0] (image) at (0,0) {\includegraphics[width=0.5\columnwidth,clip,trim={0.5in, 0.0in, 0.85in, 0.2in}]{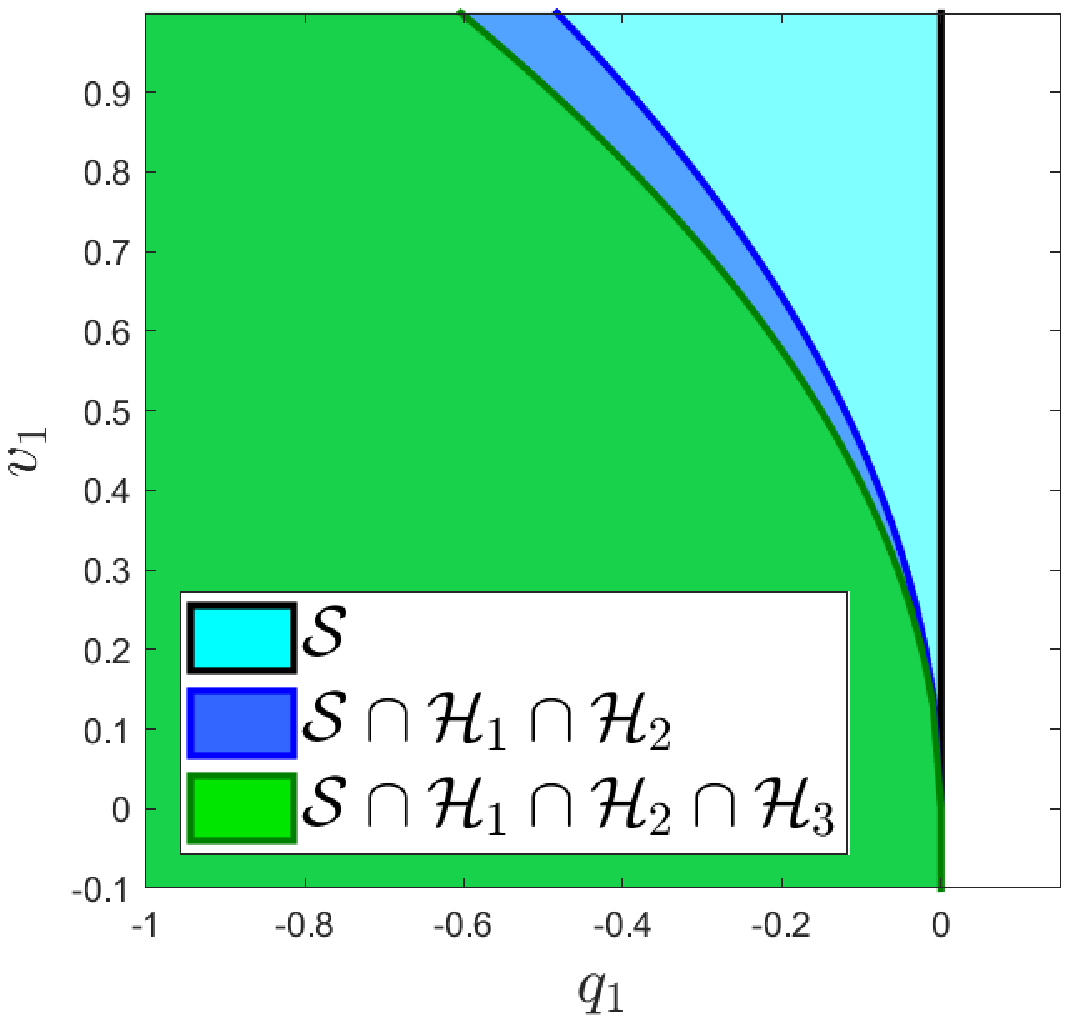}}; 
			\node [anchor=south west] (note) at (0.45,4.2) {Allowable $(q_1,v_1)$ when};
			\node [anchor=south west] (note) at (1.25,3.8) {$q_2=v_2=0$}; 
		\end{tikzpicture}
	}
		\regularversion{\vspace{-14pt}}\caption{
			Left: The CBFs $h_1,h_2$ do not satisfy Definition~\ref{def:noninterfere}, so Theorems~\ref{thm:less_than_90}-\ref{thm:less_than_90_qep} do not apply. To remedy this, we add a third constraint $\kappa_3$ and associated CBF $h_3$ to remove the remaining states where \eqref{eq:the_qp} is infeasible. Right: Fixing $q_2=v_2=0$, we show the allowable states $(q_1,v_1)$ according to 1) the safe set (cyan), 2) the CBFs in Example~\ref{ex:bad_greater_than_90} (blue), and 3) the CBFs in Example~\ref{ex:fixed_greater_than_90} (green).
		}
		\label{fig:three_constraints}
		\regularversion{\vspace{-10pt}}
	\end{figure}

	Example~\ref{ex:fixed_greater_than_90} improves upon Example~\ref{ex:bad_greater_than_90} by further limiting the system's velocity $v_1$ so as to remove all the points where \eqref{eq:the_qp} is infeasible. This can be done either by adjusting $h_1$ and $h_2$, or by adding the additional CBF $h_3$. The first approach amounts to a joint tuning of all CBFs, which is challenging, so we instead focus on generalizing the latter strategy.
	
	\begin{figure}
	\vspace{-12pt}
	\begin{algorithm}[H]
	\begin{algorithmic}[1]
		\fontsize{9.85}{11.8}\selectfont
		\Require $\mathcal{S}$ in \eqref{eq:safe_set}, $\mathcal{U}'$ possess{\color{black}ing} QEP (or OEP) w.r.t. $\mathcal{U}$, CBFs (or SCBFs) $\{h_k\}_{k=1}^{M_0}$ for $(\mathcal{S},\mathcal{U}')$ satisfying Lemma~\ref{lemma:all_invariance}
		\State $\mathcal{X} \gets \mathcal{S}\cap (\cap_{k=1}^M \mathcal{H}_{k})$
		\State $\mathcal{D} = \cup_{i\in[M_0],j\in[M_0],i\neq j}(\partial \mathcal{H}_i\cap\partial \mathcal{H}_j)$
		\State $M \gets M_0$
		\State $\mathcal{E} \gets \textrm{getInfeasibleSet}(\mathcal{D})$
		\While{$\mathcal{E}\neq\emptyset$}
		\State $\mathcal{E}_c \gets \textrm{getCluster}(\mathcal{E})$
		\State $h_{M+1} \gets \textrm{getCBF}(\mathcal{E}_c,\mathcal{X},\mathcal{U}')$
		\State $\mathcal{X} \gets \mathcal{X}\cap \mathcal{H}_{M+1}$ 
		\State $M \gets M+1$
		\State $\mathcal{D} = \cup_{i\in[M],j\in[M],i\neq j}(\partial \mathcal{H}_i\cap\partial \mathcal{H}_j)$
		\State $\mathcal{E} \gets \textrm{getInfeasibleSet}(\mathcal{D})$
		\EndWhile
		\State $\mathcal{A} \gets \mathcal{X}$
		\State \Return $\mathcal{A}$, $\{h_k\}_{k=1}^M$
	\end{algorithmic}
	\caption{Get Viability Domain}
	\label{alg:add_cbfs}
	\end{algorithm}
	\regularversion{\vspace{-24pt}}\extendedversion{\vspace{-16pt}}
	\end{figure}

	To this end, we propose Algorithm~\ref{alg:add_cbfs}. Here, $\{h_k\}_{k\in[M]}$ is a working set of CBFs and $\mathcal{X}$ is a working domain. 
	Let $\mathcal{D}$ be a set of candidate points where $\eqref{eq:the_qp}$ could be infeasible, namely all points where at least two CBFs are active as in \eqref{eq:active_set}, as Proposition~\ref{prop:qp} guarantees feasibility of \eqref{eq:the_qp} at all other points. Next, let $\mathcal{E}$ be the subset of $\mathcal{D}$ where \eqref{eq:the_qp} is actually infeasible, where $\mathcal{E}$ is more expensive to compute than $\mathcal{D}$. 
	Next, let $\textrm{getCluster()}$ be a function that divides all the points in $\mathcal{E}$ into clusters (e.g. see Fig.~\ref{fig:attitude}), and then returns a single cluster $\mathcal{E}_c\subseteq\mathcal{E}$ for focus. Given this cluster, the function $\textrm{getCBF()}$ then determines a new CBF $h_{M+1}$ for $(\mathcal{X},\mathcal{U}')$ such that the cluster $\mathcal{E}_c$ is entirely outside the CBF set $\mathcal{H}_{M+1}$. By Assumption~\ref{assumption}, such a CBF $h_{M+1}$ will always exist. Note also the use of the set $\mathcal{U}'$ satisfying Definition \ref{def:oep} or \ref{def:qep} in $\textrm{getCBF()}$; this is done to reduce the number of points where $h_{M+1}$ might conflict with the existing CBFs $\{h_k\}_{k\in[M]}$. Algorithm~\ref{alg:add_cbfs} then adds $h_{M+1}$ to the working set of CBFs and shrinks the working domain to $\mathcal{X}\cap\mathcal{H}_{M+1}$ (e.g. see Fig.~\ref{fig:three_constraints}). The \textbf{while} block then repeats until either all the CBFs are jointly feasible or the domain $\mathcal{X}$ becomes empty.
	
	Note that the exact mechanics of grouping points into clusters during $\textrm{getCluster()}$ and of generating CBFs during $\textrm{getCBF()}$ will be specific to the system under study. 
	An example of these steps is described in Section~\ref{sec:simulations}. 
	
	\regularversion{\vspace{-3pt}}\begin{proposition} \label{prop:algorithm}
		If Algorithm~\ref{alg:add_cbfs} converges, then $\mathcal{A}$ is a viability domain as in Definition~\ref{def:viable} and the controller \eqref{eq:the_qp} is always feasible and renders $\mathcal{A}$ forward invariant. 
	\end{proposition}
	\regularversion{\vspace{-3pt}}

	Algorithm~\ref{alg:add_cbfs} is similar to typical viability domain computation algorithms \cite{power_systems,zonotopes_sampled_data,viability_definitions,expanding_polytope_ci,parameterized_linear_rci,viability_kernel_recrusive_improved,barrier_domain_of_attraction}, except that the set $\mathcal{A}$ is parameterized by a collection of CBFs. Since these CBFs are defined for a control set $\mathcal{U}'$ possessing the QEP (or OEP), we can reduce the number of points that we need to check for infeasibility to only the points where 1) multiple CBFs are active as in \eqref{eq:active_set} and 2) the active CBFs violate the condition in Definition~\ref{def:noninterfere}. We next illustrate the implementation of Algorithm~\ref{alg:add_cbfs} by example on a nonlinear system.

	\extendedversion{
	\begin{remark} \label{rem:choices}
		The choice of the functions $\textrm{getCluster()}$ and $\textrm{getCBF()}$ in Algorithm~\ref{alg:add_cbfs} will also affect the conservativeness of the resulting set $\mathcal{A}$ compared to the viability kernel, and the time of convergence of the algorithm. A very small cluster size could result in many added CBFs and an explosion in computation. Note also that Algorithm~\ref{alg:add_cbfs} can be either coded and run autonomously, or followed step-by-step ``by hand'' by a practiced engineer.
	\end{remark}
	}

	\section{Application to Orientation Control} \label{sec:simulations}

	Let $q \in \{ q \in \reals^4 \mid \|q\|_2 = 1 \}$ be the quaternion describing the orientation of a spherically symmetric rigid body with unit moments of inertia, and let $\omega =(\omega_1,\omega_2,\omega_3)\in \reals^3$ be the angular velocity of the body. The dynamics are
	\begin{equation}
		\dot{q} = \frac{1}{2} \begin{bmatrix}
			0 & \omega_3 & -\omega_2 & \omega_1 \\ \omega_3 & 0 & \omega_1 & \omega_2 \\ \omega_2 & -\omega_1 & 0 & \omega_3 \\ -\omega_1 & -\omega_2 & -\omega_3 & 0
		\end{bmatrix} q, \; \dot{\omega} = u \,. \label{eq:attitude_dynamics}
	\end{equation}
	Suppose a sensitive instrument faces towards an axis $\hat{a}$ fixed to the body, and must avoid pointing towards the fixed directions $\hat{b}_1,\hat{b}_2$ in Fig.~\ref{fig:attitude} by angles $\theta_1,\theta_2$ (red cones in Fig.~\ref{fig:attitude}). These constraints can be expressed as $\kappa_1(q) = q\transpose P(\hat{a},\hat{b}_1,\theta_1) q$ and $\kappa_2(q) = q\transpose P(\hat{a},\hat{b}_2,\theta_2) q$, where $P$ is constant with respect to the state $(q,\omega)$ and is constructed as in \cite{Mesbahi2004}. Let $\mathcal{U} = \{ u \in \reals^3 \mid \|u\|_\infty \leq u_\textrm{max}\}$ for some $u_\textrm{max}$. Assume the angular velocity is bounded by $\eta(\omega) = \|\omega\|_\infty - \omega_\textrm{max}$, where the $\infty$-norm can be expressed as 6 continuously differentiable constraint functions $\{\eta_j\}_{j\in[6]}$. 
	
	First, we find a set $\mathcal{U}'$ with the OEP with respect to $\mathcal{U}$. Example~\ref{ex:oep} tells us that $\mathcal{U}' = \{ u\in\reals^3 \mid \|u\|_2 \leq \frac{u_\textrm{max}}{\sqrt{3}}\}$ is one such set, from which we derive SCBFs $h_{\kappa,1},h_{\kappa,2}$ for $(\mathcal{S},\mathcal{U}')$ given by $h_{\kappa,i}(q,\omega) = \dot{\kappa}_i(q,\omega) - \sqrt{-\beta_i \kappa_i(q)}$ for some $\beta_i\in\reals_{>0}$. We note that each $\eta_j$ already satisfies the definition of an SCBF on $(\mathcal{S},\mathcal{U}')$, so denote these SCBFs as $h_{\eta,j} \equiv \eta_j$. The eight SCBFs together satisfy the conditions of Lemma~\ref{lemma:all_invariance}. 
	{\color{black}We now} focus only on the $h_{\kappa,1},h_{\kappa,2}$ SCBFs. If the corresponding sets $\mathcal{H}_{\kappa,1},\mathcal{H}_{\kappa,2}$ satisfy $\partial \mathcal{H}_{\kappa,1} \cap \partial \mathcal H_{\kappa,2} = \emptyset$, then $\mathcal{A}$ in Lemma~3 is a viability domain and we are done. {\color{black}Here}, we chose $\theta_1,\theta_2$ sufficiently large that this does not hold\footnote{All code and parameters used can be found at \url{https://github.com/jbreeden-um/phd-code/tree/main/2023/ACC\%20Multiple\%20Control\%20Barrier\%20Functions}}, as \regularversion{\color{black}shown}\extendedversion{indicated} by the intersection of the red cones in Fig.~\ref{fig:attitude}. 
		

	\begin{figure}
		\centering
		\begin{tikzpicture}
			\node[anchor=south west,inner sep=0] (image) at (0,0) {\includegraphics[width=\columnwidth,clip,trim={0.2in, 0.2in, 0.4in, 0.46in}]{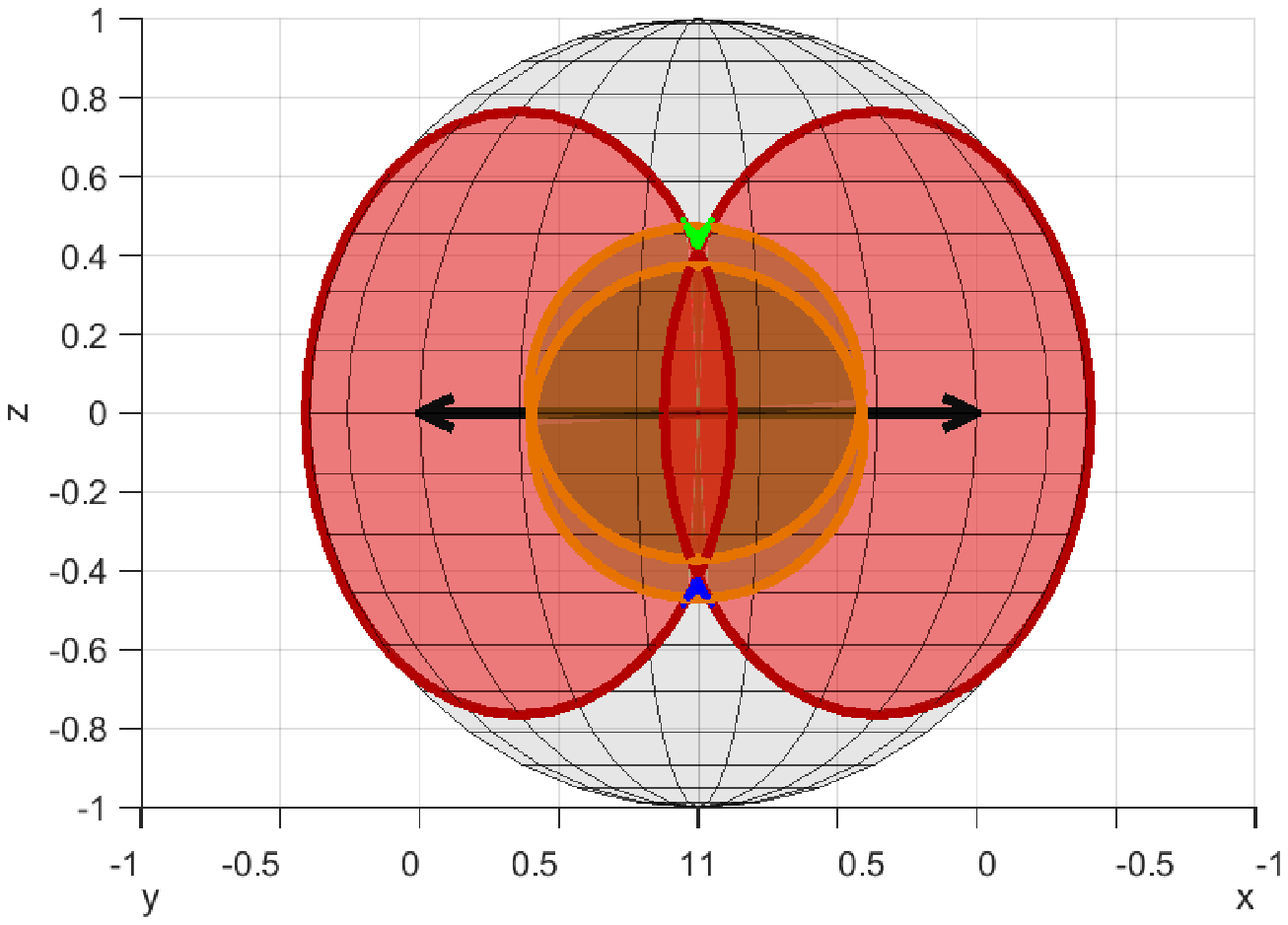}}; 
			\node [anchor=south west] (note) at (2.35,3.05) {$\hat{b}_1$};
			\node [anchor=south west] (note) at (6.5,3.05) {$\hat{b}_2$};
			\node [anchor=south west] (note) at (4.7,4.4) {\color{green}\bf Cluster 1};
			\node [anchor=south west] (note) at (4.7,1.9) {\color{blue}\bf Cluster 2};
		\end{tikzpicture}
		\regularversion{\vspace{-14pt}}\caption{
			Visualization of the constrained reorientation problem studied in Section~\ref{sec:simulations}. The large red cones are the initial unsafe states (where the body cannot point body-fixed vector $\hat{a}$), the two clusters are the states where \eqref{eq:the_qp_conditions} is infeasible, and the brown cones are the states excluded by the new CBFs introduced by Algorithm~\ref{alg:add_cbfs}. 
		}
		\label{fig:attitude}
		\regularversion{\vspace{-10pt}}\extendedversion{\vspace{-6pt}}
	\end{figure}

	We then apply Algorithm~\ref{alg:add_cbfs}. For this system, we coded $\textrm{getCluster()}$ to return connected subsets of $\mathcal{E}$. In this case, $\textrm{getInfeasibleSet()}$ returned states $(q,v)\in\reals^7$ with quaternions $q$ corresponding to the green and blue chevrons in Fig.~\ref{fig:attitude}, and $\textrm{getCluster()}$ identified these points as two clusters. We then coded $\textrm{getCBF()}$ to remove the cluster states by introducing a new constraint of the form $\kappa_{M+1}(q) = q\transpose P(\hat{a},\hat{b}_{M+1},\theta_{M+1}) q$, and computing an SCBF $h_{\kappa,M+1}$ of the same form as $h_{\kappa,1}$ and $h_{\kappa,2}$. The $\textrm{getCBF()}$ function searched for the orientation $\hat{b}_{M+1}$ and minimum angle $\theta_{M+1}$ that removed the entire cluster and that satisfied Definition~\ref{def:noninterfere} when paired with each of the existing SCBFs. Visually, for states with $\omega=0$, the new SCBF $h_{\kappa,M+1}$ resulted in removing one of the brown cones in Fig.~\ref{fig:attitude}. After removing the first cluster, Algorithm~\ref{alg:add_cbfs} repeated for a second loop, and then completed and returned the two initial SCBFs (red cones) and two new SCBFs (brown cones) shown in Fig.~\ref{fig:attitude}.

	\extendedversion{One can also implement Algorithm~\ref{alg:add_cbfs} by hand as in Remark~\ref{rem:choices}. For this system, a practiced engineer might instead decide that both clusters in Fig.~\ref{fig:attitude} should be grouped into a single cluster, which can then be removed by adding a single new SCBF representing a larger cone.}
	
	Note that the computation time of Algorithm~\ref{alg:add_cbfs} depends most on the computation time of the $\textrm{getInfeasibleSet()}$ step, which has to search a potentially large set for infeasibilities. However, this set is smaller than in typical verification algorithms, because we only have to consider states where the boundaries of two or more CBFs intersect. All other states already satisfy Lemma~\ref{lemma:nagumo} because $\mathcal{X}$ is parameterized by CBFs. Possible implementations of $\textrm{getInfeasibleSet()}$ include \cite{compatibility_checking,viable_sublevel_sets}, and the wider viability theory literature. Since these are primarily sampling-based algorithms, the computation time is dependent on the sampling interval, which depends on the Lipschitz constants of $h_k$ and the margin by which each $h_k$ satisfies \eqref{eq:control_set}. The latter is equivalent to the chosen amount of conservatism with which the CBFs are implemented; more conservative margins (e.g. smaller~$\beta_i$) {\color{black}will allow for sparser} sampling and quicker computations. In this case, the results in Fig.~\ref{fig:attitude} took 1052 seconds to compute, and results with a 10x coarser sampling took 0.30 seconds to compute, using a 3.5 GHz processor.

	\section{Conclusion} \label{sec:conclusions}
	
	
	
	We have presented conditions under which multiple CBFs are guaranteed to be jointly feasible in the presence of prescribed input bounds, while allowing each CBF to be designed in a one-at-a-time fashion under more conservative assumptions on the available control authority. However, even under these assumptions, it is still possible for there to exist points where an optimization-based controller with multiple constraints may become infeasible. Thus, we also introduced an algorithm to iteratively remove such points from the allowable state set using additional CBFs until this set becomes a viability domain and the safe controller is always feasible. Future work in this area may include robust and sampled-data extensions of this framework, {\color{black}generalizations to other classes of systems,} or extensions to relate this algorithm to the maximal viability kernel.
	
	\bibliographystyle{ieeetran}
	\bibliography{sources}

\end{document}